\documentclass[11pt,reqno]{amsart}

\usepackage[text={155mm,230mm},centering]{geometry}

\usepackage{graphicx}

\usepackage{epstopdf}
\usepackage{subfigure}
\usepackage{latexsym,bm}

\usepackage{amsmath,amsthm,amssymb,amsfonts}

\usepackage[mathscr]{eucal}

\usepackage[all]{xy}

\usepackage{mathrsfs}

\usepackage{color}

\usepackage{hyperref}

\usepackage{marginnote}
\let\oldmarginpar\marginpar
\renewcommand\marginpar[1]{\-\oldmarginpar[\raggedleft\footnotesize\color{red} #1]%
{\raggedright\footnotesize\color{red} #1}} 
\newtheorem{theorem}{Theorem}[section]
\newtheorem{lemma}[theorem]{Lemma}

\newtheorem{corollary}[theorem]{Corollary}
\newtheorem{proposition}[theorem]{Proposition}
\newtheorem{proposition-definition}[theorem]{Proposition-Definition}

\theoremstyle{definition}
\newtheorem{example}[theorem]{Example}
\newtheorem{definition}[theorem]{Definition}
\newtheorem{definition-lemma}[theorem]{Definition-Lemma}
\newtheorem{definition-theorem}[theorem]{Definition-Theorem}

\newtheorem*{convention}{Convention}

\newtheorem*{ack}{Acknowledgements}

\allowdisplaybreaks

\title[On the graded singularity 
category, I]{On the graded singularity category of
Abelian quotient singularities, I. 
Smooth categorical compactification}

\author{Xiaojun Chen}
\author{Jieheng Zeng}

\address{Chen: School of Mathematics,
Sichuan University, Chengdu,
P.R. China
and Department of Mathematics, New Uzbekistan University, Tashkent, Uzbekistan}
\email{xjchen@scu.edu.cn}

\address{Zeng: School of Mathematics and Statistics,
Peking University, Beijing, P.R. China}
\email{zengjh662@163.com}

\date{}

\begin{document}

\begin{abstract}
Given a singularity which is the quotient of an
affine space $V$ by a finite Abelian group
$G \subseteq \mathrm{SL}(V)$, we study
the DG enhancement $\mathcal{D}^{b}
(\mathrm{tails}(k[V]^G))$ of the
bounded derived category
of the non-commutative projective space $\mathrm{tails}(k[V]^G)$
and the
DG enhancement
$\mathcal{D}_{sg}^{\mathbb{Z}}(k[V]^G)$ of its graded singularity category. 
In this paper, we construct smooth categorical 
compactifications, in the sense of Efimov, 
of $\mathcal{D}^b(\mathrm{tails}(k[V]^G))$ and
$\mathcal{D}_{sg}^{\mathbb{Z}}(k[V]^G)$ respectively, 
via the non-commutative crepant resolution of $k[V]^G$.
We give
explicit constructions of canonical classical generators in the kernels of such compactifications.
\end{abstract}

\maketitle

\tableofcontents

\section{Introduction}\label{In}

\subsection{Backgrounds and motivations}
In algebraic geometry, the Abelian category of
coherent sheaves on a scheme captures a lot of information about the scheme itself. For example, the famous result of
Gabriel says that the
category of coherent sheaves 
on a scheme completely determines
the scheme itself. 
Even more strongly,
the result
of Bondal and Orlov says that a smooth Fano or anti-Fano variety is completely
determined by its bounded derived category of coherent sheaves (\cite{BO}).

In non-commutative algebraic geometry, differential graded algebras or
their module categories are often viewed by mathematicians
as non-commutative schemes
(see, for example, \cite{BVdB}). 
Given a non-negatively graded algebra $A$, 
viewed as a non-commutative scheme,
its {\it non-commutative projective space} is the 
Serre quotient category 
$ \mathrm{tails}(A) := \mathrm{grmod}(A)/ \mathrm{tors}(A)$,
where $\mathrm{tors}(A)$ consists of 
finitely generated graded
torsion modules over $A$. 
Denote by
$\pi: \mathrm{grmod}(A) \rightarrow  \mathrm{tails}(A)$
the natural projection, then
$\pi(A)$ is usually viewed as
the {\it non-commutative structure sheaf} of $\mathrm{tails}(A)$. 
As an analogue to 
the result of Gabriel, 
in \cite{AZ} Artin and Zhang showed
that some right Noetherian algebras
are determined by their non-commutative projective spaces together with some extra data.

Another subject in non-commutative geometry that has been
widely studied is the {\it singularity category} of associative algebras.
It was introduced by Buchweitz in \cite{RB} and later was
rediscovered by Orlov in \cite{DR,DR1,DR2}.
Given an associative algebra $S$, 
its {\it  singularity category}
$D_{sg}(S)$ 
is the 
triangulated quotient category 
$D^{b}(S)/\mathrm{Perf}(S)$
of the bounded derived category of $S$-modules by its subcategory
of perfect complexes. When $S$ is a Gorenstein ring, this category is 
triangle equivalent to 
the stable category of Cohen-Macaulay $S$-modules.  
It is straightforward to see that
$S$ is homologically smooth if and only if $D_{sg}(S)$ is trivial.

In the past two decades, singularity categories 
have found many important applications
in singularity theory, (non-commutative) algebraic geometry, 
representation theory and mathematical physics.
For example,
if $S$ is a
hypersurface defined by
a nontrivial function $f$, then its singularity category is triangle
equivalent to the category of
matrix factorizations
with potential $f$. Another example is that, 
the {\it graded} singularity category 
$D_{sg}^{\mathbb Z}(S) 
:= D^{b}(\mathrm{grmod}(S))/\mathrm{Perf}(\mathrm{grmod}(S))$
of $S$,
by Orlov's semi-orthogonal decomposition theorem,
is equivalent to the Kuznetsov component of $\mathrm{Proj}(S)$, if 
$S$ is generated by graded one elements with non-negative Gorenstein parameter.

Now let us assume
$S$ is a Noetherian commutative graded Gorenstein normal 
ring with isolated singularities.
For example, $S$ is the invariant subring of $k[x_1, \cdots, x_n]$
with canonical grading  
under the action of a finite group $G \subseteq \mathrm{SL}(n, k)$ (see \cite{IT}).  
The canonical DG enhancements of
the derived categories
$D^b(\mathrm{tails}(S))$
and $D_{sg}^{\mathbb Z}(S)$, just to list a few,
\begin{itemize}
\item[$-$] are both proper;
\item[$-$] both admit a Serre functor, and
\item[$-$] both admit tilting objects.
\end{itemize}
A natural question is
whether these 
nice properties can
be generalized to the {\it non-isolated} case.

It turns out that describing these two derived categories for
a non-isolated singularity is a rather challenging problem.
For example, the corresponding DG enhancements of 
$D^{b}(\mathrm{tails}(S))$ and $D_{sg}^{\mathbb Z}(S)$ 
are not proper 
when $S$ is a graded non-isolated singularity, which 
makes its homological properties very bad. Another difficulty
is that
we cannot find a homologically
smooth graded algebra 
whose non-commutative projective space is equivalent to the one of $S$. 
This poses significant challenges in investigating the two   
triangulated categories through homologically smooth algebras.
In a word, at present there is so far no  
effective way to study the general graded non-isolated singularities; see, however,
\cite{IW} and \cite{BYu}.


\subsection{Main results of the paper}

In this paper, we concentrate on a special class of
non-isolated singularities.
Let $k$ be a field of characteristic zero. Let $V=k^n$ 
for some $n \in \mathbb{N}$, and $R=k[V]$.
Let $G \subseteq \mathrm{SL}(V)$ be a 
finite Abelian group, which acts on $V$ naturally. Our primary object 
is the singularity given by $S:= R^G$, which in the literature
is called an {\it Abelian
quotient singularity}.

Let $\Lambda := G \sharp R$
be the skew group algebra, which is an Artin-Schelter regular algebra.
According to Van den Bergh (\cite{V2}), $\Lambda$ is a 
{\it non-commutative crepant resolution} of $S$. 
Both $ D^{b}(\mathrm{tails}(\Lambda))$ and 
$D^{b}(\mathrm{tails}(S))$
have a unique DG enhancement, denoted by 
$\mathcal{D}^{b}(\mathrm{tails}(\Lambda))$
and $\mathcal{D}^{b}(\mathrm{tails}(S))$ respectively 
(see \cite[Theorem A]{CNS}). 
In the meantime, $D_{sg}^{\mathbb Z}(S)$ also
admits a canonical 
DG enhancement, denoted by $\mathcal{D}_{sg}^{\mathbb Z}(S)$,  
which is given by the DG quotient 
$\mathcal{D}^{b}(\mathrm{grmod}(S))/ \mathcal{P}\mathrm{erf}(\mathrm{grmod}(S))$. 
Here $\mathcal{D}^{b}(\mathrm{grmod}(S))$ is the natural DG enhancement of 
$D^{b}(\mathrm{grmod}(S))$ and $\mathcal{P}\mathrm{erf}(\mathrm{grmod}(S))$ 
is the DG enhancement of $\mathrm{Perf}(\mathrm{grmod}(S))$. 

Now by Auslander's theorem,
there exists an $S$-module $N$ such that
$\Lambda \cong \mathrm{End}_{S}(S \oplus N)$ and $S \oplus N \cong R$ as $S$-modules.
Let $e \in \Lambda$ be the 
idempotent 
corresponding to the summand $S$.
Then we have the following functor, 
called the {\it localization functor},   
$$
(-)\otimes_{\Lambda}^{\mathbb{L}} 
\Lambda e: D^{b}(\mathrm{tails}(\Lambda)) 
\rightarrow D^{b}(\mathrm{tails}(S)),
$$
which lifts to a DG functor 
$$
\mathcal{D}^{b}(\mathrm{tails}(\Lambda)) 
\rightarrow \mathcal{D}^{b}(\mathrm{tails}(S)), 
$$
which we also denote by $(-)\otimes_{\Lambda}^{\mathbb{L}} 
\Lambda e$.  

Our strategy to study 
$\mathcal{D}^{b}(\mathrm{tails}(S))$ and 
$\mathcal{D}^{\mathbb{Z}}_{sg}(S)$ 
is via 
$\mathcal{D}^{b}(\mathrm{tails}(\Lambda))$ and 
$\mathrm{Ker}((-)\otimes_{\Lambda}^{\mathbb{L}} \Lambda e )$. 
The first result of this paper is:

\begin{theorem}[Theorem \ref{thm:smoothcatcpt1equiv}]
\label{thm:smoothcatcpt1}
The DG functor
$$
 (-) \otimes^{\mathbb{L}}_{\Lambda} 
 \Lambda e: \mathcal{D}^{b}(\mathrm{tails}(\Lambda)) 
\rightarrow \mathcal{D}^{b}(\mathrm{tails}(S))
$$
is a smooth categorical compactification, 
with the kernel given by the fully faithful DG functor
$$
F: \mathcal{D}^{b}_{\mathrm{tails}(\Lambda/\Lambda e \Lambda)}
(\mathrm{tails}(\Lambda)) 
\rightarrow \mathcal{D}^{b}(\mathrm{tails}(\Lambda)),
$$
where 
$\mathcal{D}^{b}_{\mathrm{tails}(\Lambda/\Lambda e \Lambda)}(\mathrm{tails}
(\Lambda))$ is the 
DG enhancement of 
$D^{b}_{\mathrm{tails}(\Lambda/\Lambda e \Lambda)}(\mathrm{tails}
(\Lambda))$ (see \S \ref{sjcbnhu} below
for the definition of the notations).
\end{theorem}

In the above theorem, the notion of {\it smooth categorical compactification}
of a DG category is introduced by Efimov in \cite{Efi1}, which will be recalled in
Definition \ref{def:smoothcatcomp} below. 
In \cite{Efi1}, Efimov proved that for a separable scheme $Y$ 
of finite type over a field
of characteristic zero,
its bounded derived category of coherent sheaves 
$D^b_{coh}(Y)$
admits a smooth categorical compactification, 
which verifies a conjecture of 
Kontsevich.
Theorem \ref{thm:smoothcatcpt1} generalizes his result to 
certain class of non-commutative projective spaces. 

The significance of Efimov's theorem, as shown in \cite{Efi1},
is that if a small DG category
admits a smooth categorical compactification, then it 
is smooth and homotopically 
finitely presented (hfp),
in the sense of To\"en and Vaqui\'e
(\cite{TV}); and more importantly,
there exists a concrete and computable DG algebra
such that the DG category of its perfect complexes is quasi-equivalent 
to the original DG category. 

Going back to Theorem \ref{thm:smoothcatcpt1},
the key point in its proof is to show that
$D^{b}_{\mathrm{tails}(\Lambda/\Lambda e 
\Lambda)}(\mathrm{tails}(\Lambda))$ 
is classically generated by a single object. 
The theorem is as follows.

\begin{theorem}[Theorem \ref{thm:smoothcatcpt2equiv}]\label{thm:smoothcatcpt2}
$D^{b}_{\mathrm{tails}(\Lambda/\Lambda e \Lambda)}(\mathrm{tails}(\Lambda))$ 
has a classical generator.
\end{theorem}

In Theorem
\ref{thm:smoothcatcpt2equiv}
below, we shall give
an explicit construction of
this generator.
We next move to the singularity category $D_{sg}^{\mathbb{Z}}(k[V]^G)$. 
By applying Orlov's semi-orthogonal theorem
(see Theorem \ref{orlov} below) to
$D^{b}(\mathrm{tails}(S))$, there is a localization functor
$$
\Xi: D^{b}(\mathrm{tails}(S)) \rightarrow D^{\mathbb{Z}}_{sg}(S)
$$
and hence a localization functor
$$
\Theta :=\Xi \circ ((-) \otimes^{\mathbb{L}}_{\Lambda} 
\Lambda e ): D^{b}(\mathrm{tails}
(\Lambda)) \rightarrow D^{\mathbb{Z}}_{sg}(S).
$$
Now by the proof of \cite[Theorem 16]{DR2}, 
$\Xi$ lifts to a DG functor
$\tilde{\Xi}: \mathcal{D}^{b}(\mathrm{tails}(S)) 
\rightarrow \mathcal{D}_{sg}^{\mathbb{Z}}(S)$,
and hence we obtain a localization DG functor
 $\tilde{\Theta} : \mathcal{D}^{b}(\mathrm{tails}
(\Lambda)) \rightarrow \mathcal{D}^{\mathbb{Z}}_{sg}(S)$.

\begin{theorem}[Theorem \ref{thm:catcptofsingequiv}]\label{thm:catcptofsing}
The above functor
$$
\tilde{\Theta}: \mathcal{D}^{b}(\mathrm{tails}(\Lambda)) 
\rightarrow \mathcal{D}_{sg}^{\mathbb{Z}}(S)
$$
is a smooth categorical compactification of $\mathcal{D}_{sg}^{\mathbb{Z}}(S)$. 
\end{theorem}

Thus as a corollary to Efimov's theorem,
the two DG categories $\mathcal{D}^{b}(\mathrm{tails}(S))$ and 
$\mathcal{D}_{sg}^{\mathbb{Z}}(S)$ are both homotopically finitely presented (see Corollary
\ref{cor:main1and2}).

\subsection{Structure of the paper}
The rest of the paper is devoted to the proofs of the above theorems. 
It is organized as follows:

In \S\ref{AN}, we recall some necessary 
definitions and properties
of Artin-Schelter regular algebras, skew group algebras
and their derived categories.

In \S\ref{CT}, we first 
review some basic notions
and properties of classical generators of a DG category,
and describe the classical generators of 
$D^{b}(\mathrm{tails}(\Lambda))$ and the related triangulated categories. In
the second part, we recall 
the definition of a smooth categorical compactification
and discuss some of its properties.
 
In \S\ref{subsect:DGcat}, we first describe the relations 
among $\mathcal{D}^{b}(\mathrm{tails}(\Lambda))$, 
$\mathcal{D}^{b}(\mathrm{tails}(S))$, and $\mathcal{D}_{sg}^{\mathbb{Z}}(S)$. 
After that, we give a precise description of 
the kernel of the functor $(-) \otimes_{\Lambda}^{\mathbb{L}}\Lambda e$
and then prove
Theorem \ref{thm:smoothcatcpt1}.

In \S\ref{Co}, we show Theorem \ref{thm:smoothcatcpt2},
and in \S\ref{Fina}, we show Theorem \ref{thm:catcptofsing}. 

In \S\ref{sect:HFP},
we show the above two categories
are homotopically finitely presented.

In \S\ref{EA}, we give three examples for our results. 

\begin{convention}
Throughout this paper,
$k$ denotes the base field of characteristic $0$.
All tensors and Homs are over $k$ unless otherwise specified. 
All representations and actions are faithful. Unless otherwise 
specified, all modules are right modules.  
\end{convention}

\begin{ack}We would like to 
thank 
Farkhod Eshmatov, Leilei Liu,
Song Yang 
and Shizhuo Zhang for some helpful communications.
This work is supported by NSFC Nos. 12271377 and 12261131498.
\end{ack}

\section{Artin-Schelter regular algebras and skew group algebras}\label{AN}

In this section, we recall some notions and their properties
that will be frequently used in later sections.
We first recall the definitions of Artin-Schelter regular algebras and
skew group algebras, which are highly related to each 
other and are the two main notions in this paper. After that, 
we introduce the notions of non-commutative 
projective space and 
the singularity category
of a graded algebra. 
We discuss in some detail
these categories in the case of quotient singularities.

\subsection{Artin-Schelter regular algebras}\label{subsect:ASregularalg}

Suppose $A$ is a graded associative $k$-algebra.
For $M \in \mathrm{Ob}(\mathrm{grmod}(A))$
a finitely generated graded $A$-module, we write 
$M = \bigoplus_{i} M_{i}$ as a graded vector space, 
where $M_{i}$ is the degree $i$ component of $M$; also, we denote by $M(j)$ 
the graded $A$-module such that $M(j)_{i} = M_{i+j}$. 
For any $M, N \in \mathrm{Ob}(\mathrm{grmod}(A))$, let
$$
\mathrm{\underline{Ext}}_{A}^{i}(M, N) : = \bigoplus_{j} 
\mathrm{Ext}_{\mathrm{grmod}(A)}^{i}(M, N(j)),
$$
for $i \in \mathbb{Z}$.

\begin{definition}[\cite{AS}]
Let $A$ be a locally finite 
non-negatively graded algebra
with
$A_0$ being semi-simple over $k$.	
$A$ is called an {\it Artin-Schelter (AS) regular algebra} of 
dimension $d$ with
Gorenstein parameter $a$ 
if
its global dimension is $d$,  
$$
\mathrm{\underline{Ext}}^{i}_{A}(A_{0}, A) = 0, \quad i\neq d,
$$	
and
$$
\mathrm{Ext}^{d}_{\mathrm{grmod}(A)}(A_{0}, A(j)) 
\cong
\left\{
\begin{array}{cl}0,&j \neq -a,\\
A_{0},&j=-a.
\end{array}
\right.
$$
as $A_0$-modules. 
\end{definition}

In the above definition,
instead of the global dimension, 
if the injective dimensions
of $A$ as both left and right $A$-modules are $d$,
then we call $A$ a {\it Gorenstein algebra} of  
dimension $d$ with
Gorenstein parameter $a$. 

\begin{example}
By using Koszul resolution, it is quite straightforward to check that 
the polynomial ring
$R:=k[x_1,\cdots, x_n]$ is an AS regular algebra of dimension $n$
with Gorenstein parameter $n$. 
\end{example}

Let $S$ be a Noetherian local ring, $\mathfrak{m}$ be the unique maximal 
ideal of $S$
and $M$ be a finitely generated $S$-module. Recall that the 
{\it depth} of $M$, denoted by $\mathrm{depth}(M)$,
is defined as follows:
\begin{enumerate}
\item If $\mathfrak{m}M = M$, then $\mathrm{depth}(M) = \infty$;

\item If $\mathfrak{m}M \neq M$, then $\mathrm{depth}(M)$ is the 
supremum in $\mathbb{Z}$ of the lengths of
$M$-regular sequences in $\mathfrak{m}$.
\end{enumerate}
Here, an $M$-regular sequence means a sequence elements 
$\{f_r \}_{1 \leq r \leq m}$ of $S$ such that for any 
element $f_r$ in this sequence, 
\begin{enumerate}
\item[(1)] $f_r$ is not a zero divisor on 
$M/ M(f_{1}, \cdots, f_{r-1})$, and 
\item[(2)] $M/ M(f_{1}, \cdots, f_{m})$
is not trivial as an $S$-module. 
\end{enumerate}

\begin{definition}
Let $S$ be a commutative Noetherian ring with Krull 
dimension $d$ and $M$ be a finitely generated $S$-module.
$M$ is called {\it Maximal Cohen-Macaulay} if 
$\mathrm{depth}(M_{\mathfrak{m}}) = d$
for any maximal ideal $\mathfrak{m}$ of $S$, 
or $M \cong 0$. 
In what follows we simply call such
a module Cohen-Macaulay. 

If $S$ is a graded commutative Noetherian ring, 
then a {\it graded} Cohen-Macaulay $S$-module is a 
Cohen-Macaulay $S$-module endowing a grading 
compatible with the grading of $S$. 
\end{definition}

\begin{proposition}
Let $S$ be a commutative Gorenstein ring over $k$. A finitely 
generated $S$-module $M$
is Cohen-Macaulay if and only if
$
\mathrm{Hom}_{S}(M, R[i]) = 0,
$
for any $i > 0$.
\end{proposition}
If $M$ is a Cohen-Macaulay $S$-module, then $M$ is also reflexive, i.e., 
$$
\mathrm{Hom}_{S}(\mathrm{Hom}_{S}(M, S), S) \cong M
$$ 
as $S$-modules.

By the above proposition, it is easy to see that any finitely generated
projective $S$-module is Cohen-Macaulay.

\begin{example}[\cite{IT}]
Let $G$ be a finite subgroup in $\mathrm{SL}(V)$ and
$W$ be an irreducible representation of $G$ over $k$. 
Then $(W \otimes k[V])^{G}$ is a Cohen-Macaulay 
$k[V]^{G}$-module, where $G$ acts on $W \otimes k[V]$ diagonally. 
\end{example}

\subsection{Skew group algebras}\label{NCSC}

Now let $G$ be a finite subgroup in $\mathrm{GL}(V)$ and let $R:=k[V]$. 
Recall that the {\it skew group algebra}, denoted by $G\sharp R$, is given
as follows:
\begin{enumerate}
\item[(1)] $G\sharp R$ is isomorphic to $kG \otimes R$ as $k$-vector spaces, 
where $kG$ is the group algebra of $G$; 

\item[(2)] For any $(g_{1}, f_1), (g_{2}, f_2) \in G\sharp R$,
their product
$
(g_{1}, f_1)\cdot (g_{2}, f_2) = (g_{1}g_{2}, 
\, g_{2}^{-1} (f_1) f_{2}) \in G\sharp R
$. 
\end{enumerate}
The following result is nowadays standard.

\begin{proposition}[\cite{MR}] \label{ASstru}
Let $G$ be a finite subgroup in $\mathrm{GL}(V)$,
and $R:=k[V]=k[x_1,\cdots, x_n]$
Then the skew group algebra $G \sharp R$ is an AS regular algebra of
dimension $n$ with Gorenstein parameter $n$.  
\end{proposition}

In what follows, for a finite-dimensional $G$-representation $W$, 
denote by $M^{G}_{R}(W)$ the $R^{G}$-module 
$(W \otimes R)^{G}$. 
The following result is proved by Iyama and Takahashi:

\begin{proposition}[{\cite[Theorem 4.2]{IT}}]
\label{prop:Lambda}
Denote $S = R^G$. 
Let $\hat{G}$ to be the set of irreducible representations of $G$ and 
$\Lambda:= 
\mathrm{End}_{S}(\bigoplus\limits_{W \in \hat{G}} M^{G}_{R}(W)^{\oplus |W|})$, 
where $|W| = \mathrm{dim}(W)$. Suppose that $G \subseteq \mathrm{SL}(V)$. 
Then $\Lambda\cong G\sharp R$ as algebras. 
\end{proposition}

In what follows, we usually use $e$ to denote the idempotent of $\Lambda$ 
which corresponds to the summand
$S = R^G \subseteq \bigoplus\limits_{W \in \hat{G}} M^{G}_{R}(W)^{\oplus |W|}$.

\begin{example}
Let $R = k[x_{1}, x_{2}]$ with the gradings of $x_1$ and $x_2$ both being one.
Let
$
\sigma :=\begin{pmatrix}
-1 & 0  \\
0 &-1
\end{pmatrix}
$
and $G := \langle\sigma\rangle\subset\mathrm{SL}(2, k)$. 
Then $S = R^G = k[x^{2}_{1}, x^{2}_{2}, x_{1} x_{2}]
\cong k[A, B, C]/ (AB - C^2)$ with $\mathrm{deg}(A) 
= \mathrm{deg}(B) = \mathrm{deg}(C) = 2$.
Note $G$ only has two irreducible representations. 
One is the one-dimensional trivial representation $k$,
and the other is the one-dimensional representation $W$ 
such that $\forall x \in W$, $\sigma(x) = -x$.
Thus we obtain that
$\Lambda = \mathrm{End}_{S}(S \oplus (M^{G}_{R}(W))\cong G \sharp R$,
which is an AS regular algebra of dimension 2 with Gorenstein parameter 2.
Moreover, $\Lambda$ can be written as a quiver algebra $kQ_{\Lambda}/ I$, where
$Q_{\Lambda}$ is given as follows:
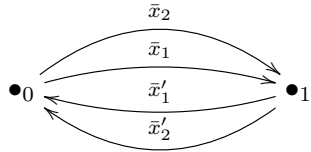
\begin{displaymath}
\xymatrix{
\bullet_{0} \ar@/^0.3cm/[rrr]^{\bar{x}_{1}} \ar@/^0.8cm/[rrr]^{\bar{x}_{2}}
&&& \bullet_{1} \ar@/^0.3cm/[lll]_{\bar{x}_{1}'} \ar@/^0.8cm/[lll]_{\bar{x}_{2}'}
}
\end{displaymath}
where
vertex 0 corresponds to the summand $S$,
vertex 1 corresponds to the summand $M^{G}_{R}(W)$,
the arrows $\bar{x}_{1}$, $\bar{x}_{2}$ correspond to the two nontrivial morphisms
$$
x_{1}, x_{2} \in \mathrm{Hom}_{S}(S, M^{T}_{R}(W))
\cong \mathrm{Hom}_{\mathrm{mod}(k, G)}(k, R \otimes W) \subseteq R
$$
respectively, the arrows $\bar{x}_{1}'$, $\bar{x}_{2}'$ correspond to the nontrivial morphisms
$$
x_{1}, x_{2} \in  \mathrm{Hom}_{S}(M^{T}_{R}(W), S)
\cong \mathrm{Hom}_{\mathrm{mod}(k, G)}(W, R) \subseteq R
$$
respectively, and $I$ is generated by elements 
$\bar{x}_{1}\bar{x}_{2}'- \bar{x}_{2}\bar{x}_{1}'$ 
and 
$\bar{x}_{1}'\bar{x}_{2} - \bar{x}_{2}'\bar{x}_{1}.$
\end{example}

\subsection{Non-commutative projective space}
We next study
the non-commutative projective space
and the singularity category
with some detail.

\subsubsection{Quotient of a triangulated category}
We start with the Serre sub- and quotient category of  
a triangulated category.

\begin{definition}\textup{(1)}
Let $\mathcal{A}$ be an Abelian category and $\mathcal{B}$ 
be a full Abelian subcategory of $\mathcal{A}$. 
Then $\mathcal{B}$ is called a {\it Serre subcategory} 
of $\mathcal{A}$ if for any short exact sequence
$$
0 \longrightarrow M \longrightarrow N \longrightarrow W \longrightarrow 0,
$$ in $\mathcal{A}$,
$N \in \mathrm{Ob}(\mathcal{B})$ if and only if 
$M, W \in \mathrm{Ob}(\mathcal{B})$.

\textup{(2)}
Let $\mathcal{A}$ be an Abelian category and $\mathcal{B}$ be a 
Serre subcategory of $\mathcal{A}$. 
The {\it Serre quotient category} of $\mathcal{A}$ by $\mathcal{B}$, denoted by
$\mathcal{A}/\mathcal{B}$,
is an Abelian category given as follows:
\begin{enumerate}
\item[\textup{(i)}] The objects of $\mathcal{A}/\mathcal{B}$ are the same with $\mathcal{A}$;

\item[\textup{(ii)}] For any $M, N \in \mathrm{Ob}(\mathcal{A}/\mathcal{B})$,
$$
\mathrm{Hom}_{\mathcal{A}/\mathcal{B}}(M, N)
:= \mathrm{Hom}_{\mathcal{A}}(M, N)/I_{M, N},
$$
where $I_{M, N}$ is $k$-linear subspace of $\mathrm{Hom}_{\mathcal{A}}(M, N)$ 
consisting of morphisms
which factors through some objects in $\mathcal{B}$.
\end{enumerate}
\end{definition}

\begin{definition}
Let $\mathcal{C}$ be a triangulated category and 
$\mathcal{D}$ be a full triangulated subcategory of $\mathcal{C}$.
The {\it triangulated quotient category} (or called {\it Verdier localization}) 
of $\mathcal{C}$ by $\mathcal{D}$, denoted by $\mathcal{C}/\mathcal{D}$,
is a triangulate category given as follows:
\begin{enumerate}
\item The objects of $\mathcal{A}/\mathcal{B}$ are the same with $\mathcal{A}$;

\item For any $M, N \in \mathrm{Ob}(\mathcal{C}/\mathcal{D})$,
$$
\mathrm{Hom}_{\mathcal{C}/\mathcal{D}}(M, N) := \mathrm{Hom}_{\mathcal{C}}(M, N)[\mathcal{S}^{-1}],
$$
where $\mathcal{S}^{-1}$ denotes the inverses of morphisms in 
$\mathcal{S} \subset \mathrm{Hom}_{\mathcal{C}}(M, N)$, 
where
$\mathcal{S} := \{f \in 
\mathrm{Hom}_{\mathcal{C}}(M, N) \,|\, \mathrm{cone}(f) \in \mathrm{Ob}(\mathcal{D}) \}$.
\end{enumerate}
\end{definition}
Note that there is a natural functor $\mathcal{C} \rightarrow 
\mathcal{C}/\mathcal{D}$ called 
the triangulated quotient. 

\begin{definition}\label{SeDef}
Let $\mathcal{A}$ be an Abelian category and $\mathcal{B}$ be a 
Serre subcategory of $\mathcal{A}$,
and let $D^{\ast}(\mathcal A)$ be 
the bounded, or bounded from
above, or bounded from below, or unbounded,
derived category of $\mathcal A$.
Then
$D^{\ast}_{[\mathcal{B}]}(\mathcal{A})$ 
is defined to be  
the full triangulated subcategory of $D^{\ast}(\mathcal{A})$
consisting of objects whose cohomologies are all in 
$\mathrm{Ob}(\mathcal{B})$. 
\end{definition}

It is obvious that there is a natural functor $\mathcal{A} \rightarrow \mathcal{A}/\mathcal{B}$,
given by the projection.
Moreover, it induces a
functor $D^{\ast}(\mathcal{A}) \rightarrow D^{\ast}(\mathcal{A}/\mathcal{B})$. 
The following proposition describes the relationship between the above 
two types of quotient
categories.

\begin{proposition}[{\cite[Theorem 3.2]{MJI}}]\label{serre}
With the above notions,
the above functor $D^{\ast}(\mathcal{A}) \rightarrow D^{\ast}(\mathcal{A}/\mathcal{B})$
induces an equivalence of categories
$$
 D^{\ast}(\mathcal{A})/D^{\ast}_{[\mathcal{B}]}(\mathcal{A}) \cong D^{\ast}
 (\mathcal{A}/\mathcal{B}).
$$  
\end{proposition}

\subsubsection{Non-commutative projective space and singularity category}

We next recall the non-commutative projective space and the 
singularity category of a non-negatively graded algebra.

\begin{definition}
Suppose $A$ is a non-negatively graded algebra.
\begin{enumerate}

\item The {\it non-commutative projective space} of $A$, denoted by 
$\mathrm{tails}(A)$,
is the Serre quotient category $\mathrm{grmod}(A) / \mathrm{tors}(A)$,
where $\mathrm{grmod}(A)$ is the Abelian category of finitely 
generated graded $A$-modules, 
and $\mathrm{tors}(A)$ is the full subcategory of $\mathrm{grmod}(A)$ 
consists of finitely generated graded torsion $A$-modules, i.e.,
$$
\mathrm{tors}(A) := \big\{M \in \mathrm{Ob}(\mathrm{grmod}(A))
 \mid M A_{\geq i} = 0, \,\, \mbox{for some} \,\, i   \big\}. 
$$

\item The {\it graded  singularity category}
of $A$, denoted by 
$D_{sg}^{\mathbb{Z}}(A)$,
is the triangulated quotient category $D^{b}( \mathrm{grmod}(A) 
) / \mathrm{Perf}( \mathrm{grmod}(A) )$,
where $\mathrm{Perf}(\mathrm{grmod}(A))$ is the homotopy category of 
bounded complexes of 
finitely generated graded projective $A$-modules.
\end{enumerate}
\end{definition}

\subsubsection{Application to $\Lambda$}\label{sjcbnhu}

Now, for any non-negatively graded algebra $A$ and its graded quotient 
algebra $A' = A/I$, 
let $\overline{\mathrm{grmod}}(A') \subseteq \mathrm{grmod}(A)$ 
be the full subcategory such that 
$$
\mathrm{Ob}( \overline{\mathrm{grmod}}(A') ) 
= \big\{ M \in \mathrm{Ob}(\mathrm{grmod}(A) ) \mid M I^{i} = 0, \,\,
\mbox{for some} \,\, i \in \mathbb{N}  \big\}, 
$$
and $\overline{\mathrm{tails}}(A') \subseteq \mathrm{tails}(A)$ be the 
full subcategory such that 
$$
\mathrm{Ob}( \overline{\mathrm{tails}}(A') ) 
= \big\{ M \in  \mathrm{Ob}(\mathrm{tails}(A) ) \mid M I^{i} = 0, \,\,
\mbox{for some}\,\, i \in \mathbb{N}  \big\}. 
$$ 

\begin{lemma}\label{ToAbel}Retain the above notations.
\begin{enumerate}\item[$(1)$]
$\overline{\mathrm{grmod}}(A')$ is a Serre subcategory of 
$\mathrm{grmod}(A)$, and

\item[$(2)$]
$\overline{\mathrm{tails}}(A')$ is a Serre subcategory of 
$\mathrm{tails}(A)$.
\end{enumerate}
\end{lemma}

\begin{proof}
(1) Let 
$$
0 \longrightarrow N \stackrel{\iota}
\longrightarrow M \stackrel{p}
\longrightarrow W \longrightarrow 0 
$$ 
be a short exact sequence in $\mathrm{grmod}(A)$.  
 If $M$ is an object of $\overline{\mathrm{grmod}}(A')$, then it satisfies that 
 $M I^i = 0$ 
 for some $i \in \mathbb{N}$, 
 which is equivalent to that it is an object 
 of $\mathrm{grmod}(A / I^i )$. On one hand, 
 it follows that $N I^i = 0$ 
 since $N \subseteq M$ implies that $ N I^i \subseteq M I^i$. 
 Hence, $N$ is an object of $\overline{\mathrm{grmod}}(A')$. 
On the other hand, $M I^i = 0$ implies that 
\begin{align*}
W I^i  = (M / N) I^i 
 = M I^i /  (M I^i \cap N ) 
 = 0. 
\end{align*}
It follows that $W$ is also an object of $\overline{\mathrm{grmod}}(A')$. 

If $N$ and $W$ are both objects of $\overline{\mathrm{grmod}}(A')$, 
we assume that they are both objects of some 
$\mathrm{grmod}(A / I^j )$. Now, we have $p(M)I^j = 0$. 
Hence, $p(M I^j) = 0$.
It follows that $M I^j$ is in the image of $\iota$. Since $N \cong \iota(N)$ 
as graded $A$-modules, 
we have $\iota(N) I^j = 0$, and then obtain that $M I^{2j} =  (M I^{j})I^{j}= 0$. 
It implies that $M$
is an object of $\overline{\mathrm{grmod}}(A')$. In summary, 
$\overline{\mathrm{grmod}}(A')$ is a Serre categories of
$\mathrm{grmod}(A)$. 

(2) The proof is completely analogous, and is left to the interested reader.
\end{proof}

From Definition \ref{SeDef}, we introduce the following triangulated categories:
$$
D^{b}_{\mathrm{grmod}(A')}(\mathrm{grmod}(A)) 
: = D^{b}_{[\overline{\mathrm{grmod}}(A')]}(\mathrm{grmod}(A)) 
$$
and 
$$
D^{b}_{\mathrm{tails}(A')}(\mathrm{tails}(A)) 
: = D^{b}_{[\overline{\mathrm{tails}}(A')]}(\mathrm{tails}(A)).
$$

In what follows, we apply the above
notions to the case where $A$
is $\Lambda\cong G\sharp R$
(see Proposition \ref{ASstru}).
Also let $I$ in the above be $\Lambda e\Lambda$, where 
$I^i = I$ for any $i$. It is direct to see that 
$$
\mathrm{grmod}(\Lambda / \Lambda e \Lambda) 
\cong \overline{\mathrm{grmod}}(\Lambda / \Lambda e \Lambda), \,\, \quad 
\mathrm{tails}(\Lambda / \Lambda e \Lambda) 
\cong  \overline{\mathrm{tails}}(\Lambda / \Lambda e \Lambda).
$$ 
We then
have the following. 

\begin{proposition}\label{Ker1}
\textup{(1)}
For the functor 
$(-) \otimes^{\mathbb{L}}_{\Lambda} \Lambda e: 
D^{b}(\mathrm{grmod}(\Lambda)) \rightarrow D^{b}(\mathrm{grmod}(S))
$, 
$$
\mathrm{Ker}((-) \otimes^{\mathbb{L}}_{\Lambda} \Lambda e) 
\cong D^{b}_{\mathrm{grmod}(\Lambda / \Lambda e \Lambda)}(\mathrm{grmod}(\Lambda)).
$$

\textup{(2)}
For the functor 
$
(-) \otimes^{\mathbb{L}}_{\Lambda} \Lambda e: D^{b}(\mathrm{tails}
(\Lambda)) \rightarrow D^{b}(\mathrm{tails}(\Lambda / \Lambda e \Lambda)) 
$, 
$$
\mathrm{Ker}((-) \otimes^{\mathbb{L}}_{\Lambda} \Lambda e )
 \cong D^{b}_{\mathrm{tails}(\Lambda / \Lambda e \Lambda)}(\mathrm{tails}(\Lambda)).
$$
\end{proposition}

\begin{proof}
(1) By \cite[Proposition 5.9]{MJI}, 
the functor $(-) \otimes_{\Lambda} \Lambda e$
induces an equivalence
$$
\mathrm{grmod}(\Lambda ) / \mathrm{grmod}(\Lambda / \Lambda e \Lambda) \cong \mathrm{grmod}(S)
$$
of Abelian categories. 
By Proposition \ref{serre}, 
it then induces an equivalence
$$
D^{b}(\mathrm{grmod}(\Lambda)) / D^{b}_{\mathrm{grmod}(\Lambda / \Lambda e \Lambda)}
(\mathrm{grmod}(\Lambda)) \cong D^{b}(\mathrm{grmod}(S))
$$
as triangulated categories.
It follows that 
$$
\mathrm{Ker}((-) \otimes^{\mathbb{L}}_{\Lambda} \Lambda e) 
\cong D^{b}_{\mathrm{grmod}(\Lambda / \Lambda e \Lambda)}(\mathrm{grmod}(\Lambda))
$$
as triangulated categories.

(2) The proof is the same as (1)
and is therefore omitted.
\end{proof}

\section{Smooth categorical compactification}
\label{CT}

The goal of this section is to introduce Efimov's
notion of smooth categorical compactification of 
DG categories (see Definition \ref{def:smoothcatcomp}). 
We shall also
recall some necessary
concepts and properties of triangulated categories,
which will be used in later sections.

\subsection{Classical generators of a triangulated category}

Let $\mathcal{I}_1$ and $\mathcal{I}_2$ be two full triangulated subcategories of $\mathcal{T}$.
Denote by $\mathcal{I}_1 \ast \mathcal{I}_2 $ the full subcategory of $\mathcal{T}$,
whose objects consist of objects, say $M$, such that there exists a distinguished triangle 
$$
M_{1} \longrightarrow M \longrightarrow M_2
$$
in $\mathcal{T}$ with $M_{i} \in \mathrm{Ob}(\mathcal{I}_i)$. 
Let $\mathcal{E}$ be a set of objects in $\mathcal{T}$;
denote by $\langle\mathcal{E}\rangle :=\langle\mathcal{E}\rangle_1$ 
the smallest full subcategory of 
$\mathcal{T}$ containing the objects in $\mathcal{E}$ and closed under 
direct summands, finite direct sums and shifts. 
Let $\langle\mathcal{E}\rangle_0$ be the trivial subcategory of $\mathcal{T}$,
 $\langle\mathcal{E}\rangle_{i} := \big<\langle\mathcal{E}\rangle_{i-1} 
 \ast\langle\mathcal{E}\rangle\big>$, and $\langle
 \mathcal{E}\rangle_{\infty} := \bigcup_{i \geq 0}\langle\mathcal{E}\rangle_{i}$ 
 as full subcategory.

\begin{definition}
Let $\mathcal{T}$ be a triangulated category.
A set $\mathcal{E}$ of objects 
in $\mathcal{T}$ is said to {\it classically generate} $\mathcal{T}$ if
$$
\mathcal{T} = \langle\mathcal{E}\rangle_{\infty}.
$$
$\mathcal{E}$ is called a {\it classical generator} of $\mathcal{T}$. 
\end{definition}

Equivalently, a set $\mathcal{E}$ of objects in $\mathcal{T}$  classically generates
$\mathcal{T}$ if for any object $M$ in $\mathcal{T}$, $M \in \langle 
\mathcal{E}\rangle_{n}$ for some $n \in \mathbb{N}$. We say that $M$
is {\it classically generated} by 
$\mathcal{E}$ in $\mathcal{T}$ if $M \in \langle\mathcal{E}\rangle_{n}$ 
for some $n\in \mathbb{N}$.

\subsection{Orlov's semi-orthogonal decomposition}\label{SG}

Let $\mathcal{T}$ be a triangulated category and $\mathcal{A}$ be 
a full triangulated subcategory of $\mathcal{T}$. 
The full triangulated subcategory $\mathcal{A}^{\perp}$
(resp. $^{\perp}\mathcal{A}$) of $\mathcal{T}$ 
consists of objects, say $M$, of $\mathcal T$ such that 
$
\mathrm{Hom}_{\mathcal{T}}(N, M) = 0
$,
for any $N \in \mathrm{Ob}(\mathcal{A})$
(resp. $\mathrm{Hom}_{\mathcal{T}}(M, N)$ for any $N \in \mathrm{Ob}(\mathcal{A})$). 

\begin{definition}
A sequence of 
full triangulated subcategories $\mathcal{A}_{1}, \cdots, \mathcal{A}_m$ of 
$\mathcal{T}$ is said to be a
{\it semi-orthogonal collection} if  
$$
\mathrm{Ob}(\mathcal{A}_j) \subseteq \mathrm{Ob}(\mathcal{A}^{\perp}_i)
$$
for any $i > j$. 
A semi-orthogonal collection $\mathcal{A}_{1}, \cdots, \mathcal{A}_m$ 
of $\mathcal{T}$ is said to be a 
{\it semi-orthogonal decomposition} of $\mathcal{T}$ 
if for any object $X$ of $\mathcal{T}$ there are morphisms:
$$
0 = X_m \rightarrow X_{m-1} \rightarrow \cdots\rightarrow X_{1} \rightarrow X_0 = X
$$
in $\mathcal{T}$ such that for any $k \in [1, m]$,
$$
\mathrm{Cone}(X_k \rightarrow X_{k-1}) \in \mathrm{Ob}(\mathcal{A}_k).
$$	
A semi-orthogonal decomposition is usually 
written as
$$
\mathcal{T} = \big<\mathcal{A}_{1}, \mathcal{A}_{2}, \cdots, \mathcal{A}_m \big>.
$$
\end{definition}

For a non-negatively graded Gorenstein algebra,
the following theorem gives a relationship between
the bounded derived category of its
non-commutative projective space and its graded singularity category,
which is usually called {\it Orlov's semi-orthogonal decomposition theorem}:

\begin{theorem}[{Orlov \cite[Theorem 16]{DR2}}]\label{orlov}
Suppose $A$ is a non-negatively graded Gorenstein algebra of global dimension 
$d$ with Gorenstein parameter $a > 0$ 
such that $A_0$ is semi-simple. 
If $A$ is Noetherian, 
then there is a fully faithful functor 
$\Phi: D^{\mathbb{Z}}_{sg}(A) \rightarrow D^{b}(\mathrm{tails}\, A)$
and the following semi-orthogonal decomposition
\begin{align}\label{siucno}
D^{b}(\mathrm{tails}(A)) = \big< \pi(A), \pi (A(1)), \cdots, 
\pi (A(a-1)), \Phi (D^{\mathbb{Z}}_{sg}(A)) \big>,
\end{align}
where $\pi: D^b(\mathrm{grmod}(A))
\to D^b(\mathrm{tails}(A))$
is the natural quotient functor.
\end{theorem}
In the decomposition (\ref{siucno}), we usually use 
$\pi (A(i)$ to represent $\langle\pi (A(1)\rangle$ 
for simplicity.

Now let $A$ be an AS regular algebra with Gorenstein parameter $a$. 
Then $D^{\mathbb{Z}}_{sg}(A)$ is a trivial category since $A$ is homologically smooth. 
Hence, the above theorem implies that in this case
$$
D^{b}(\mathrm{tails}(A)) = \big< \pi(A), \pi (A(1)), \cdots, 
\pi (A(a-1))\big>.
$$
In summary, we have the
following.

\begin{corollary}\label{ydbvwjv}
Suppose $A$ is an AS regular algebra with positive Gorenstein parameter $a$. Then
$D^{b}(\mathrm{tails}(A))$ has a classical generator 
$
\bigoplus\limits^{a-1}_{i = 0}\pi(A(i)).
$
\end{corollary}

\subsection{Smooth categorical compactification}

We now recall the definition and some properties of DG categories. 
For more details, we refer the reader to \cite{KB0}.

\begin{definition}
A {\it DG category} is a $k$-linear category $\mathcal{D}$ 
such that the Hom-set $\mathrm{Hom}_{\mathcal{D}}(X, Y)$ 
consists of complexes of vector spaces, the  composition  
$$
\mathrm{Hom}_{\mathcal{D}}(Y, Z) \otimes 
\mathrm{Hom}_{\mathcal{D}}(X, Y) \rightarrow \mathrm{Hom}_{\mathcal{D}}(X, Z),
$$ 
is a morphism of complexes and the identity morphisms are 
closed in degree zero.

A {\it DG functor} $F: \mathcal{D} \rightarrow \mathcal{D}'$ between DG categories is a 
functor such that the map 
$$
F_{X, Y}: \mathrm{Hom}_{\mathcal{D}}(X, Y) \rightarrow \mathrm{Hom}_{\mathcal{D}'}(F(X), F(Y))
$$
is a morphism of complexes. 
\end{definition}

\begin{definition}
A DG functor $F:D\to D'$ is called a {\it quasi-equivalence} if 
for any $X, Y \in \mathrm{Ob}(H^{0}(\mathcal{A}) )$, the morphism 
$$
F_{X, Y}: \mathrm{Hom}_{\mathcal{D}}(X, Y) \rightarrow \mathrm{Hom}_{\mathcal{D}'}(F(X), F(Y))
$$
of complexes is a quasi-isomorphism and $F$ is essentially  surjective. 
\end{definition}

Denote by $\mathrm{DGCat}$ the category whose objects are small 
DG categories and whose morphisms are DG functors. 
Consider the localization, denoted by $\mathrm{Hqe}$, of  $\mathrm{DGCat}$ 
with respect to quasi-equivalences. 
A morphism in $\mathrm{Hqe}$ will be called a {\it DG quasi-functor}. 
We next introduce the definition of
DG enhancement. 

For any DG category $\mathcal{A}$, let $\mathrm{Mod}(\mathcal{A})$ be the DG category of DG modules over
$\mathcal{A}$.
Denote by
$l_{\mathcal{A}}: \mathcal{A} \hookrightarrow \mathrm{Mod}(\mathcal{A})$ the Yoneda embedding
(see \cite[\S1.2]{KB0}).

\begin{definition}[\cite{BKKK}]
Let $\mathcal{A}$ be a DG category. A pre-triangulated structure on $\mathcal{A}$ is given as follows: 

\begin{enumerate}
\item[(1)] For any morphism $f \in Z^0(\mathrm{Hom}_{\mathcal{A}}(X, Y) )$,
the {\it cone} of $f$, denoted by $
\mathrm{Cone}(f) \in \mathrm{Ob}(\mathcal{A})$, is the
$\mathcal{A}$ DG-module 
$$
\mathrm{Cone}( \mathrm{Hom}_{\mathcal{A}}(-, X) 
\xrightarrow{f_\ast} \mathrm{Hom}_{\mathcal{A}}(-, Y)  )
$$ under the functor $l_{\mathcal{A}}$; 

\item[(2)] For any $i \in \mathbb{Z}$ and $X \in \mathrm{Ob}(\mathcal{A})$, 
the {\it $i$-translation} of $X$, denoted by $X[i]$, is the 
$\mathcal{A}$ DG-module 
$$
\mathrm{Hom}_{\mathcal{A}}(-, X)[i]
$$
under the functor $l_{\mathcal{A}}$. 
\end{enumerate}	
\end{definition}

\begin{definition}
Let $\mathcal{A}$ be a DG category. 
Its homotopy category 
$H^{0}(\mathcal{A})$ is defined as follows: 
\begin{enumerate}
\item [(1)] $\mathrm{Ob}(H^{0}(\mathcal{A}))$ is same with 
$\mathrm{Ob}(\mathcal{A})$;

\item [(2)] For any two $X, Y \in \mathrm{Ob}(H^{0}(\mathcal{A}) )$, 
$\mathrm{Hom}_{H^{0}(\mathcal{A})}(X, Y) 
:= H^0 (\mathrm{Hom}_{\mathcal{A}}(X, Y) )$.
\end{enumerate}
\end{definition}
Note that when $\mathcal{A}$ admits a pre-triangulated structure, 
this structure gives a triangulated structure 
for $H^{0}(\mathcal{A})$ (see \cite[\S 2.1 ]{KB1}). 
Moreover, 
we view $H^{0}(-)$ as a functor from $\mathrm{DGCat}$ 
to the category consisting of all triangulated categories.  
Note that $H^{0}(-)$ takes quasi-equivalences to triangle equivalences. 

\begin{definition}
Let $\mathcal{T}$ be a triangulated category. A DG category 
$\mathcal{A}$ is called a {\it DG enhancement} of 	
$\mathcal{T}$ if 
\begin{enumerate}
\item [(1)] $\mathcal{A}$ admits a pre-triangulated structure;

\item [(2)] There is a triangle equivalent functor $\Phi: H^{0}(\mathcal{A}) \rightarrow \mathcal{T}$. 
\end{enumerate}
\end{definition}

\begin{definition}[{\cite[1.2]{DRI}}]\label{sfkd}
Let $\mathcal{A}$ be a DG category with a pre-triangulated structure, $\mathcal{B} \subseteq \mathcal{A}$ be a 
full DG  subcategory with a pre-triangulated structure compatible with the one of $\mathcal{A}$. 
A {\it DG quotient} (or simply a quotient) of $\mathcal{A}$
by
$\mathcal{B}$ is a diagram of DG 
categories and DG functors:
$$
\mathcal{A} \xleftarrow{\sim} \tilde{\mathcal{A}} \xrightarrow{\xi} \mathcal{C}
$$
such that
\begin{enumerate}
\item[(1)] $\tilde{\mathcal{A}}$ has a pre-triangulated structure 
and the DG functor $\tilde{\mathcal{A}} \xrightarrow{\sim} \mathcal{A}$ 
is a quasi-equivalence which is compatible with pre-triangulated structures; 

\item[(2)] $H^0(\xi): H^0(\tilde{\mathcal{A}}) 
\rightarrow H^0(\tilde{\mathcal{C}})$ is essentially surjective; 

\item[(3)] $\tilde{\mathcal{A}}$ has a pre-triangulated structure 
such that $H^{0}(\xi)$ induces a triangle equivalence 
$H^0(\mathcal{A}) / H^0(\mathcal{B}) \xrightarrow{\sim} H^0(\tilde{\mathcal{C}})$. 
\end{enumerate}
\end{definition}

For simplicity, we call $\mathcal{C}$ the DG quotient instead of the above diagram.  
 
\begin{proposition}[{\cite[1.6.2]{DRI}}]\label{nskkx}
For any small DG category $\mathcal{A}$ with a pre-triangulated structure, and 
a full DG subcategory $\mathcal{B} \subseteq \mathcal{A}$ with 
a pre-triangulated structure which is compatible with the one of $\mathcal{A}$, the DG quotient 
of the pair $(\mathcal{A}, \mathcal{B})$ (meaning $\mathcal{A}$
modulo $\mathcal{B}$) exists and is unique in the 
sense of $\mathrm{Hqe}$. Moreover, the DG quotient, denoted by $\mathcal{A}/\mathcal{B}$, 
is a DG enhancement of the 
triangulated category $H^{0}(\mathcal{A})/H^{0}(\mathcal{B})$. 
\end{proposition}

\begin{example}\label{jsvdic}
(1) Let $A$ be a finitely generated associative $k$-algebra. Both 
$D^{b}(A)$ and the full subcategory $\mathrm{Perf}(A)$ have unique 
DG enhancements, 
denoted by $\mathcal{D}^{b}(A)$ and $\mathcal{P}\mathrm{erf}(A)$ 
respectively (see \cite[Theorem A]{CNS}). Since the two pre-triangulated structures are compatible, 
the pair $(\mathcal{D}^{b}(A), \mathcal{P}\mathrm{erf}(A))$ give a canonical DG enhancement 
$\mathcal{D}_{sg}(A) := \mathcal{D}^{b}(A) / \mathcal{P}\mathrm{erf}(A)$ of 
$D_{sg}(A) = D^{b}(A)/\mathrm{Perf}(A)$ via the DG quotient. 

(2) Let $S$ be a finitely generated 
graded $k$-algebra 
with non-negative grading. 
We have two associated DG categories 
$\mathcal{D}^{b}(\mathrm{grmod}(S))$ and $\mathcal{P}\mathrm{erf}(\mathrm{grmod}(S))$. They also induce
a canonical DG enhancement 
$\mathcal{D}^{b}(\mathrm{grmod}(S)) / \mathcal{P}\mathrm{erf}(\mathrm{grmod}(S))$ of the 
$D^{\mathbb{Z}}_{sg}(S)$ via the DG quotient.

(3) Let $S$ be a finitely generated graded algebra over $k$ with 
non-negative grading. Consider the non-commutative projective space 
$\mathrm{tails}(S)$. Its bounded derived category has
a unique DG enhancement $\mathcal{D}^{b}(\mathrm{tails}(S)) $ 
(see also \cite[Theorem A]{CNS}). Suppose $S$ is a Gorenstein 
algebra with Gorenstein parameter $a > 0$.
Let $\big< \pi(S), \cdots, \pi(S(a-1)) \big>^{\sim} 
\subseteq \mathcal{D}^{b}(\mathrm{tails}(S))$ be the full DG subcategory whose 
objects are generated by 
objects $\{ \pi(S), \cdots, \pi(S(a-1)) \}$. Then the DG quotient given by pair 
$$
(\mathcal{D}^{b}(\mathrm{tails}(S)), \big< \pi(S), \cdots, \pi(S(a-1)) \big>^{\sim} )
$$ 
is exactly $\mathcal{D}^{b}(\mathrm{grmod}(S)) 
/ \mathcal{P}\mathrm{erf}(\mathrm{grmod}(S))$ in
$\mathrm{Hqe}$ (see the proof of \cite[Theorem 16]{DR2}). 	
\end{example}

\begin{definition}[{\cite[\S 3.7]{KL0}}] 
\begin{enumerate}
\item[(1)] A DG category $\mathcal{A}$ is called {\it smooth} 
if the diagonal $\mathcal{A} \otimes \mathcal{A}^{op}$-module $\mathcal A$ 
is perfect, that is,
$
\mathcal{A} \in \mathrm{Ob}( \mathrm{Perf}(\mathcal{A} \otimes \mathcal{A}^{op})).
$

\item[(2)] A DG category $\mathcal{A}$ is called {\it proper} if for any $X, Y \in \mathcal{A}$,
$
\mathrm{Hom}_{\mathcal{A}}(X, Y)
$ has bounded and finite-dimensional cohomology. 

\item[(3)] Let $\mathcal{A}$ be a DG category with a pre-triangulated structure. An 
object $X$ is called a {\it classical generator} 
of $\mathcal{A}$ if $X$ is a classical generator of $H^0(\mathcal{A})$. 
\end{enumerate}
\end{definition}

\begin{definition}[{\cite[Definition 1.7]{Efi1}}]\label{def:smoothcatcomp}
A {\it smooth categorical compactification} of a DG category $\mathcal{A}$ is a
DG quasi-functor $F: \mathcal{C} \rightarrow \mathcal{A}$, where the
DG category $\mathcal{C}$ is smooth and proper, 
such that the extension of scalars
functor $F^{\ast}: \mathrm{perf}(\mathcal{C}) \rightarrow \mathrm{Perf}(\mathcal{A})$ is a
localization (up to direct summands), and its kernel is classically generated by a single object.
\end{definition}

If $\mathcal{A}$ is a pre-triangulated DG quasi-functor, then there is an equivalent
definition for the smooth categorical compactification as follows.

\begin{proposition-definition}[{\cite[Definition 2.2]{Efi}}]\label{Csc}
For a pre-triangulated DG category $\mathcal{A}$, a categorical smooth compactification is a DG functor 
$F: \mathcal{C} \rightarrow \mathcal{A}$, such that: 
\begin{enumerate}
\item $\mathcal{C}$ is smooth and proper pre-triangulated DG category;
\item the induced functor $\mathcal{C} / \mathrm{Ker}(F) \rightarrow \mathcal{A}$ is fully faithful,
and $\mathrm{Ker}(F)$ is classically generated by a single object;
\item every object $X \in \mathcal{A}$ is a direct summand of $F(Y)$, for some $Y \in \mathcal{C}$.
\end{enumerate}
\end{proposition-definition}

The above proposition suggests that
\begin{itemize}
\item
if $\mathcal{C}$ is a 
smooth and proper pre-triangulated DG category and 
\item
if $F$ is a DG quotient such that the homotopy 
category of $\mathrm{Ker}(F)$ has a classical generator, 
\end{itemize}
then $F$ gives a smooth categorical compactification of $\mathcal{A}$.

\section{The non-commutative projective spaces of quotient singularities}\label{subsect:DGcat}

The purpose of this section is to prove Theorem
\ref{thm:smoothcatcpt1}, that is, to show
$
 (-) \otimes^{\mathbb{L}}_{\Lambda} \Lambda
 e : \mathcal{D}^{b}(\mathrm{tails}
 (\Lambda)) 
\rightarrow \mathcal{D}^{b}(\mathrm{tails}
(S))
$
is a categorical smooth compactification. We check that 
$ (-) \otimes^{\mathbb{L}}_{\Lambda} \Lambda e $
satisfies all the conditions given 
Proposition-Definition \ref{Csc}.

\subsection{Finite group actions on affine spaces}\label{fgas}

We first recall some basic properties of finite Abelian  subgroups in $\mathrm{SL}(V)$.
Let $G$ be a finite Abelian subgroup in $\mathrm{SL}(V)$ and $R = k[V] = k[x_1, \cdots, x_n]$ 
with $\mathrm{deg}(x_i) = 1$. 
The following proposition is straightforward. 

\begin{proposition}\label{CanAb}	
Any finite Abelian  subgroup $H$ in $\mathrm{GL}(V)$
is isomorphic to a certain subgroup 
whose elements are all diagonal matrixes.
And any irreducible representation of $H$ is one-dimensional vector space. 		
\end{proposition}

Recall from \S\ref{NCSC} that
for any irreducible representation $W$ of $G$, the graded $R^{G}$-module
$M^{G}_R(W): = (W \otimes R)^G$ is an indecomposable Cohen-Macaulay module. 
Hence, by Proposition \ref{CanAb} and Auslander's theorem, we have
$$
R \cong \bigoplus_{W \in \hat{G}} (M^{G}_R(W))^{\oplus |W|} 
\cong \bigoplus_{W \in \hat{G}} M^{G}_R(W)
$$
 as $R^G$-modules, and
$$
G \sharp R \cong \mathrm{End}_{R^{G}}(R)
$$
as algebras. 
In the meantime, $G \sharp R$ is a quiver algebra, called the McKay quiver,  
such that $e'(G \sharp R)e' \cong R^G$ 
for any indecomposable idempotent $e'$ (see \cite[2.2.1]{BFC}). 
Furthermore, for any finite Abelian subgroup $H$ in $\mathrm{GL}(V)$, $H \sharp R$ is also a
quiver algebra whose vertices are given by $kH$.

\subsection{The non-commutative projective space}\label{fgasrr}

We next introduce the following
lemma, which will be used 
in the proof of Theorem \ref{thm:smoothcatcpt1}.

\begin{lemma}\label{compgene}
Suppose 
$A$ is a non-negatively graded algebra and $B := A /I $ 
is a graded quotient algebra of $A$.
Let $\iota: A\to B$ be the projection. 
Suppose $X$ is a classical generator of $D^{b}(\mathrm{tails}(B))$.
Then $\iota_{\ast}(X)$ is a classical generator of 
$D_{\mathrm{tails}(B)}^{b}(\mathrm{tails}(A))$.
\end{lemma}

\begin{proof}
Since the functor $\iota_{\ast}$ preserves the torsion modules, $\iota_{\ast}$  
is well-defined on $\mathrm{tails}(B)$. Moreover, it is 
an exact functor, and its derived functor $\mathrm{R}\iota_{\ast} = \iota_{\ast}$ 
takes objects in 
$D_{\mathrm{tors}(B)}^{b}(\mathrm{grmod}(B))$ to 
objects in
$D_{\mathrm{tors}{A}}^{b}
(\mathrm{grmod}(A))$. 
By Proposition \ref{serre}, $\iota_{\ast}$ induces a derived functor 
$$
D^{b}(\mathrm{tails}(B)) \rightarrow D^{b}(\mathrm{tails}(A)).  
$$
Note that since $\iota_{\ast}$ is an exact functor on $\mathrm{grmod}(B)$,
the induced functor (also denoted by $\iota_{\ast}$) on $\mathrm{tails}(B)$ is exact, too. 

Without loss of
generality, let $M^{\bullet}$ be an object in  $D_{\mathrm{tails}(B)}^{b}(\mathrm{tails}(A))$ 
such that $M^{i} = 0$ when $i < 0$, or $i > r$, for some positive number $r$.
That is, $M^{\bullet}$ is in the form
$$
\cdots \longrightarrow 0 \longrightarrow M^0 \stackrel{d_0}\longrightarrow M^1 
\longrightarrow \cdots \stackrel{d_{r-1}}
\longrightarrow M^{r} \longrightarrow 0 \longrightarrow \cdots .
$$ 
We have a distinguished triangle
$$
\mathrm{H}^{0}(M^{\bullet}) 
\longrightarrow M^{\bullet} 
\longrightarrow \tilde{M}^{\ast}[-1] 
\longrightarrow \mathrm{H}^{0}
(M^{\bullet})[1],
$$
where 
$$\widetilde{M}^{\ast} := 
\cdots \longrightarrow 0 \longrightarrow M^1 / \mathrm{Im}(d_0) \stackrel{d_1}\longrightarrow 
M^2 \longrightarrow \cdots \stackrel{d_{r-1}}
\longrightarrow M^{r} \longrightarrow 0 \longrightarrow \cdots.
$$
Since $M^{\bullet}$ is an object of $D_{\mathrm{tails}(B)}^{b}(\mathrm{tails}(A))$, 
$\mathrm{H}^{0}(M^{\bullet})$ is an object of $\mathrm{tails}(A / I^j))$ 
for some $j \in \mathbb{Z}$. 

Next, 
consider the short exact sequence:
$$
0 \longrightarrow I / I^{j} \longrightarrow  A/ I^{j}   \longrightarrow  A/ I \longrightarrow 0.
$$
Tensoring it with 
$\mathrm{H}^{0}(M^{\bullet})$ over $\mathrm{grmod}(A / I^{j})$
we get the following short exact sequence:
$$
0 \longrightarrow \overline{\mathrm{H}^{0}(M^{\bullet})}  \longrightarrow \mathrm{H}^{0}(M^{\bullet}) 
\longrightarrow \mathrm{H}^{0}(M^{\bullet}) \otimes_{A / I^j} A / I \longrightarrow 0  
$$
in $\mathrm{tails}(A)$,
where
$\overline{\mathrm{H}^{0}(M^{\bullet})}$ is a graded $A$-submodule 
of $\mathrm{H}^{0}(M^{\bullet}) \otimes_{A / I^j} I / I^j$. 
Since $\mathrm{H}^{0}(M^{\bullet}) \otimes_{A / I^j} A / I$ is an object of  
$\mathrm{tails}(A / I)$ 
via the 
functor $\iota_{\ast}$, and $ \mathrm{H}^{0}(M^{\bullet}) \otimes_{A / I^{j}} I / I^{j}$ is an object of 
$\mathrm{tails}(A / I^{j - 1})$,
we obtain that
$\overline{\mathrm{H}^{0}(M^{\bullet})}$ is also an object of $\mathrm{tails}(A / I^{j - 1}))$. 
By induction on $j$, we obtain that $\mathrm{H}^{0}(M^{\bullet})$ 
is classically generated 
by objects in the essential image of the functor $\iota_{\ast}$. 

 In the meantime, since
 $X$ is a classical generator of
 $D^{b}(\mathrm{tails}(B))$, 
 we have 
 $ \mathrm{H}^{0}(M^{\bullet}) \in 
\langle \iota_{\ast}(X)\rangle_{p}$ for some $p \in \mathbb{N}$. 
Hence,
to prove that $M^{\bullet}$ is classically generated by 
$\iota_{\ast}(X)$, it 
suffices to show that $\widetilde{M}^{\ast}$ is classically generated 
by $\iota_{\ast}(X)$.  
Note that 
the length $r$ of 
$\widetilde{M}^{\ast}$ is less 
than the length of $M^{\bullet}$. 
By induction on the length of 
the complexes, we obtain that $M^{\bullet} \in \langle \iota_{\ast}(X) 
\rangle_{p_M}$ 
for some positive number $p_M$. Thus 
$\iota_{\ast}(X)$ is a classical generator of $D_{\mathrm{tails}(B)}^{b}(\mathrm{tails}(A))$.
\end{proof}

For simplicity, for any non-negatively graded algebra $A$ and 
its graded quotient algebra $B$, we use $X$ to represent 
$\iota_{\ast}(X)$ 
as an object of 
$\mathrm{tails}(A)$. 
We now rephrase 
Theorem \ref{thm:smoothcatcpt1}.

\begin{theorem}[Theorem \ref{thm:smoothcatcpt1}]\label{thm:smoothcatcpt1equiv}
The DG functor
$$
 (-) \otimes^{\mathbb{L}}_{\Lambda} \Lambda e : \mathcal{D}^{b}(\mathrm{tails}(\Lambda)) 
\rightarrow \mathcal{D}^{b}(\mathrm{tails}(S))
$$
is a smooth categorical compactification, 
with the kernel given by the fully faithful DG functor
$$
F: \mathcal{D}^{b}_{\mathrm{tails}(\Lambda/\Lambda e \Lambda)}(\mathrm{tails}(\Lambda)) 
\rightarrow \mathcal{D}^{b}(\mathrm{tails}(\Lambda)),
$$
where $e$ is the idempotent corresponding to the summand $S$ and 
$\mathcal{D}^{b}_{\mathrm{tails}(\Lambda/\Lambda e \Lambda)}(\mathrm{tails}(\Lambda))$ is the 
DG enhancement of $D^{b}_{\mathrm{tails}(\Lambda/\Lambda e \Lambda)}(\mathrm{tails}
(\Lambda))$ induced by $D^{b}_{\mathrm{tails}(\Lambda/\Lambda e \Lambda)}(\mathrm{tails}
(\Lambda)) \subseteq D^{b}(\mathrm{tails}(\Lambda))$.
\end{theorem}

To prove the theorem, let us first recall the following.

\begin{definition}
The {\it Beilinson algebra} $\nabla(A)$ of an AS regular algebra $A$ with Gorenstein parameter $a > 0$
is the algebra
$$
\begin{pmatrix}
A_0 & A_1 & \cdots & A_{a-1}  \\
0 & A_0 & \cdots & A_{a-2} \\
\vdots & \vdots & \ddots & \vdots \\
0 & 0 & \cdots & A_0
\end{pmatrix}
$$
with the multiplication of matrices. 
\end{definition}

By \cite[Theorem 4.14]{MM}, there is an isomorphism
$$
\mathcal{D}^{b}(\mathrm{tails}(G \sharp R)) \cong 
\mathcal{P}\mathrm{erf}(\nabla(G \sharp R)), 
$$
and moreover, 
$\nabla(G \sharp R)$ is a finite-dimensional algebra over $k$. Thus
$\mathcal{D}^{b}(\mathrm{tails}(G \sharp R))$ 
is a proper DG category. 
Moreover, by \cite[Proposition 7.5.1]{MR1}, $\nabla(G \sharp R)$ is a 
homologically smooth algebra, which implies that 
$\mathcal{P}\mathrm{erf}(\nabla(G \sharp R))$ 
is a smooth DG category. Thus
$\mathcal{D}^{b}(\mathrm{tails}(G \sharp R))$ is also a smooth DG category. 
We therefore get the following. 

\begin{proposition}\label{gfubjnv}
The DG category $\mathcal{D}^{b}(\mathrm{tails}(G \sharp R))$ is smooth and proper.
\end{proposition}

Now, recall that
$$
\Lambda = G \sharp R = \mathrm{End}_{S}( \bigoplus_{W \in \hat{G}} M^{G}_R(W) ). 
$$ 
For any $X \in D^{b}(\mathrm{tails}(R^G))$,  
by $e \Lambda e \cong R^G \cong S$ as algebras, 
we have
$(X \otimes_{S} e \Lambda ) \otimes^{\mathbb{L}}_{\Lambda} \Lambda e = X$.
It implies that the functor 
$$
(-) \otimes^{\mathbb{L}}_{\Lambda} \Lambda e : D^{b}(\mathrm{tails}(G \sharp R))
 \rightarrow D^{b}(\mathrm{tails}(R^G))
$$
is necessarily surjective. 
Moreover, by Proposition \ref{serre},  
this functor is a localization of triangulated categories. 
By Proposition \ref{Csc}, we get the fully faithful functor
$$
F:  D^{b}_{\mathrm{tails}(\Lambda/\Lambda e \Lambda)}(\mathrm{tails}
(\Lambda)) \cong \mathrm{Ker}( (-) \otimes^{\mathbb{L}}_{\Lambda} \Lambda e )  
\rightarrow D^{b}(\mathrm{tails}(\Lambda)). 
$$
From the uniqueness of the
DG enhancements of $D^{b}(\mathrm{tails}(G \sharp R))$ and 
$D^{b}(\mathrm{tails}(R^G))$, both $(-) \otimes^{\mathbb{L}}_{\Lambda} \Lambda e$ and $F$ 
lifts to DG functors. 
Here, the lifting of $(-) \otimes^{\mathbb{L}}_{\Lambda} \Lambda e$ and $F$ are 
DG functors which descend to
$(-) \otimes^{\mathbb{L}}_{\Lambda} \Lambda e$ and $F$ when taking 
the
homotopy categories.  
Thus, by Proposition-Definition \ref{Csc}, to prove 
Theorem \ref{thm:smoothcatcpt1equiv}, it suffices to show that
the triangulated category $\mathrm{Ker}( (-) \otimes^{\mathbb{L}}_{\Lambda} \Lambda e ) $ 
is classically  generated by a single object. 

Now by Proposition \ref{Ker1} (2), 
$$
\mathrm{Ker}( (-) \otimes^{\mathbb{L}}_{\Lambda} 
\Lambda e ) = D^{b}_{\mathrm{tails}(\Lambda/ \Lambda e \Lambda)}(\mathrm{tails}(\Lambda)),
$$
and we have the following.

\begin{theorem}[Theorem \ref{thm:smoothcatcpt2}]\label{thm:smoothcatcpt2equiv}
$D^{b}_{\mathrm{tails}(\Lambda/\Lambda e \Lambda)}(\mathrm{tails}(\Lambda))$ 
has a classical generator $X_F$; more
precisely, 
\begin{equation}
\label{eq:defofXF}
X_F = \bigoplus_{(\chi, G^i) \in \tilde{G}_0} 
\bigoplus_{j = 0}^{\mathrm{dim}(V^i) -1} G / G^i \sharp k[V^i](j)
\end{equation}
(See \S\ref{Co} for the
notations).
\end{theorem}

The proof of the above
theorem is
postponed to the next section.

\begin{proof}[Proof of Theorem \ref{thm:smoothcatcpt1equiv}
(=Theorem \ref{thm:smoothcatcpt1})]
By \cite[Theorem A]{CNS}, $D^{b}(\mathrm{tails}(S))$ admits a 
unique DG enhancement. Then according to
Definition \ref{sfkd}, the DG quotient 
$$\mathcal{D}^{b}(\mathrm{tails}
(\Lambda)) / \mathcal{D}^{b}_{\mathrm{tails}
(\Lambda /\Lambda e \Lambda)}(\mathrm{tails}(\Lambda))$$ is 
exactly this unique DG enhancement. 
Thus by Proposition \ref{nskkx},  
$(-) \otimes^{\mathbb{L}}_{\Lambda} \Lambda e$
gives this DG quotient  
$$
\mathcal{D}^{b}(\mathrm{tails}(\Lambda)) \rightarrow \mathcal{D}^{b}(\mathrm{tails}(S)) 
$$
in $\mathrm{Hqe}$. In the meantime, 
by Proposition \ref{gfubjnv}, 
$\mathcal{D}^{b}(\mathrm{tails}(\Lambda)) 
\cong \mathcal{P}\mathrm{erf}(\nabla \Lambda)$ is smooth and proper. 
Then it is left to prove that 
$\mathcal{D}^{b}_{\mathrm{tails}(\Lambda /\Lambda e \Lambda)}
(\mathrm{tails}(\Lambda))$ is classically generated by a single object,
which follows from 
Theorem \ref{thm:smoothcatcpt2equiv}. 
The proof is now complete.
\end{proof}

\section{Classical generator of $D^{b}_{\mathrm{tails}
(\Lambda/\Lambda e \Lambda)}(\mathrm{tails}(\Lambda))$}\label{Co}

In above section, we have shown that
$(-) \otimes^{\mathbb{L}}_{\Lambda} \Lambda e$ is a localization functor.
The purpose of this section is to
show that $X_F$,
given by \eqref{eq:defofXF}, is a classical generator of 
$D^{b}_{\mathrm{tails}(\Lambda/\Lambda e \Lambda)}(\mathrm{tails}(\Lambda))$, 
or equivalently, of $\mathrm{Ker}(F)$,
which then proves Theorem 
\ref{thm:smoothcatcpt2equiv}
(and hence Theorem
\ref{thm:smoothcatcpt2}).

Let us first introduce some notations.
As in the previous sections, $G$ is a finite Abelian group
of $\mathrm{GL}(V)$.
\begin{enumerate}
\item[$-$]
Let $\{ H^j \}_j$ be the set of subgroups in $G$ such that the invariant subspace of $H^j$, denoted by $W^j$. 
Furthermore, 
let $\{ K^l \}_l \subseteq \{ H^j \}_j$ be the subset
such that $K^l$ is
maximal among all subgroups that
have the same invariant subspace $W^l$. 

\item[$-$]
Let $\{(\chi, H^j) \}$ be the
set of pairs where $\chi: G \rightarrow k^{\ast} \cong k\backslash \{0\}$ 
is a character of $G$. 
Moreover,
let $\{(\chi, G^i) \} \subseteq \{(\chi, H^j) \}$ be the subset of pairs where
$G^i \in\{ K^l \}_l$ satisfies that $\chi|_{G^i}$ is nontrivial and
there is no other $K^l$ such that
$K^l \subseteq G^i$ and $\chi|_{K^l}$ is nontrivial.
Let $V^i \subseteq V$ be the invariant subspace of $G^i$. 
\end{enumerate}

\begin{definition}
\label{def:equiv}
Introduce an equivalence relation on 
$\{(\chi, H^j) \}$
as follows:
$(\chi_1, H^{j_1}) \sim (\chi_2, H^{j_2})$ if
and only if $H^{j_1} = H^{j_2}$ in $G$ and
$\chi_1|_{H^{j_1}} = \chi_2|_{H^{j_1}}$. 
Let $\tilde{G} := \{ (\chi, H^j) \} / \sim$ and
$\tilde{G}_0 := \{ (\chi, G^i) \} / \sim$,
where the equivalence relation on
$\{(\chi, G^i)\}$ is the restriction of the
one on $\{(\chi, H^j)\}$.
\end{definition}

\subsection{Sketch of the proof}\label{subsect:strategy} 

The proof of Theorem
\ref{thm:smoothcatcpt2equiv}
consists of the following steps:

\vspace{2mm}

\noindent{\bf Step 1:} 
fixing an indecomposable idempotent $e'$ with character $\chi'$ 
and a subgroup $H \subseteq G$, there is 
a canonical 
graded algebra surjection 
$\Lambda/ \Lambda e \Lambda \twoheadrightarrow G/H \sharp k[V^{H}]$ (see Proposition \ref{Modstru}).

Consider the sum of these morphisms over
$\tilde G_0$,
we hope that 
$$\Lambda/\Lambda e\Lambda\to 
\bigoplus_{(\chi, G^i) \in \tilde{G}_0}  G/G^i \sharp k[V^i]
$$
would be an injection; if this is true,
then in our case, the modules
over the left-hand side algebra
$\Lambda/\Lambda e\Lambda$
can be expressed by the modules
over the right-hand side algebra.
But this is not true, since
the kernel of the above map
may contain nilpotent elements. This leads
to the next step.

\vspace{2mm}

\noindent{\bf Step 2:}
 we introduce the ``reduced" ideal $\sqrt{\Lambda e \Lambda} \subseteq \Lambda$ of $\Lambda e \Lambda$, 
 and then algebra surjections $\Lambda/ \sqrt{\Lambda e \Lambda} \twoheadrightarrow G/H \sharp k[V^{H}]$ 
induced by the above algebra surjections. 
Next, taking the direct sum of these
surjections 
over $\tilde G_0$, we show 
$$
\Lambda/\sqrt{\Lambda e \Lambda} \rightarrow 
\bigoplus_{(\chi, G^i) \in \tilde{G}_0}  G/G^i \sharp k[V^i]
$$
is in fact an injection (see Lemma \ref{NNned1}). 

\vspace{2mm}

\noindent{\bf Step 3:} by using the injection above, 
we construct a bounded, graded $(\Lambda/ \sqrt{\Lambda e \Lambda})^e$-module 
resolution of $\Lambda/ \sqrt{\Lambda e \Lambda}$, where the components are
graded $( \bigoplus\limits_{(G^i, V^i) \in \tilde{G}_0} G/G^i \sharp k[V^i] )^{e}$-modules 
(see Proposition \ref{qjsgkn} and thereafter); 

\vspace{2mm}

\noindent{\bf Step 4:}
by analyzing the derived tensor functor given by above resolution,  
we obtain Lemma \ref{hidjks}, which says
that,
for
any $X$ in $D^{b}(\mathrm{grmod}(\Lambda / \sqrt{\Lambda e \Lambda}) )$,
$$
X \in \Big< \bigoplus\limits_{j= l_X}^{n_X} \bigoplus\limits_{(\chi, G^i) \in \tilde{G}_0}  
G / G^i \sharp k[V^i](j) \Big>_{r},
$$
for some $r \in \mathbb{N}$ and $l_X, n_X \in \mathbb{Z}$.
Next by this lemma, Corollary \ref{ydbvwjv}
and the fact that 
$\bigoplus\limits_{(G^i, V^i) \in \tilde{G}_0} G/G^i \sharp k[V^i]$ is a direct sum of 
AS regular algebras, 
we show that $$\bigoplus\limits_{(\chi, G^i) \in \tilde{G}_0} 
\bigoplus\limits_{j = 0}^{\mathrm{dim}(V^i) -1} G / G^i \sharp k[V^i](j)$$ is a
classical generator of $D^{b}(\mathrm{tails}
(\Lambda / \sqrt{\Lambda e \Lambda}))$ (see Lemma \ref{mjdbck}). 

\vspace{2mm}

\noindent{\bf Step 5:} Finally, 
we show that
any classical generator of
$D^{b}_{\mathrm{tails}(\Lambda/ \sqrt{\Lambda e \Lambda})}(\mathrm{tails}(\Lambda))$ classically generates
$D^{b}_{\mathrm{tails}(\Lambda/\Lambda e \Lambda)}(\mathrm{tails}(\Lambda))
$ (see Lemma \ref{nribe}),
and then the theorem follows.

In the rest of this section,
after some algebraic preparation given in \S\ref{subsect:algprep},
we apply the above strategy step by step,  
culminating in the proof of the
theorem.

\subsection{Some algebraic preparations}\label{subsect:algprep}

Fix a pair $(\chi, H^j)$ in $\tilde{G}$, 
we are going to 
construct an algebra morphism $k(G / H^j) \hookrightarrow kG$. 
First, recall the following well-known facts.

\begin{proposition}\label{ReChaId}
$(1)$ The indecomposable idempotents in $kG$ are in one to one 
correspondence to the characters of $G$. 		

$(2)$ $kG$ is isomorphic to the direct sum of all indecomposable 
idempotents in $kG$ as algebras. 	
\end{proposition}

In fact, given a character $\chi$,
the idempotent $e_{\chi}$ 
corresponding to 
it is
$\displaystyle\frac{1}{|G|} \sum\limits_{g \in G} (\chi(g) g )$. 
Now fix a
pair $(\chi, H^j)$ in $\tilde{G}$.
Let $\{e^{j}_{i} \}_{i}$ be the set of indecomposable idempotents in 
$kG$ satisfying $\chi|_{H^j} = \chi_{i}^j|_{H^j}$, 
where $\chi^{j}_{i}$ 
is the character
corresponding to $e^{j}_i$. 
Set $e_{0}^{j} := e_{\chi}$. 
We thus 
get a semi-simple algebra $\bigoplus_{i} ke_{i}^j$, 
which is a subalgebra of $kG$. 

Now, for any character $\lambda$ 
of $G$ such that $\lambda|_{H_j} = \{1 \}$, we have 
$\lambda\chi \in \{\chi_i^j \}_i$. Let $\{\lambda_i^j \}_i$ be the set of character of $G$ where 
$\lambda_i^{j}|_{H_j} = \{1 \}$. Then there is a one to one correspondence 
between 
\begin{equation}
\label{1-1:lambdachi}
\{\lambda_i^j \}_i\Leftrightarrow
\{\chi_i^j \}_i,    
\end{equation}
which
sends $\lambda_i^j$ to $\lambda_i^j \chi$. 
For convenience, we use $\chi^j_i$ to denote
$\lambda_i^j \chi$. 

Moreover, it is direct to see that 
there is also a one to one correspondence 
$$\{\mbox{charaters of $G/H^j$}\}
\Leftrightarrow
\{\lambda_i^j\}_i.
$$
Then, we view $\lambda_i^j$ as 
a character of $G/H^j$. 
Thus by Proposition 
\ref{ReChaId}(1), for $G/H^j$  
we obtain a one to one correspondence 
$$
\{\mbox{indecomposable
idempotents of $k(G/H^j)$}\}
\Leftrightarrow
\{\chi_i^j\}_i
$$
from \eqref{1-1:lambdachi}. 
From the one to one correspondence $\{\chi_i^j \}_i\Leftrightarrow
\{e_i^j \}_i$,  
it follows the one to one 
correspondence
$$\{\mbox{indecomposable idempotents of $k(G/H^j)$}\}\Leftrightarrow\{e_i^j \}_i.$$
Finally, by Proposition \ref{ReChaId}(2) 
for $G/H^j$, we get the following.

\begin{lemma}\label{Subalge}
Fix an indecomposable idempotent $e_{\chi}$. Then there is an algebra isomorphism
$$  
\kappa^{j}_{\chi}:  k(G / H^j) \xrightarrow{\sim} \bigoplus_{i} ke_{i}^j  
$$
given by the above one to one correspondence. 
\end{lemma}
By Lemma \ref{Subalge}, we know that there is an algebra  injection 
$$
k(G/ H^j) \xrightarrow{\kappa^{j}_{\chi}} \bigoplus_{i} ke_{i}^j 
\hookrightarrow kG 
$$
when fixing an indecomposable idempotent $e_{\chi}$. 
For simplicity, 
we also use $\kappa^{j}_{\chi}$ to denote 
the above 
injection. More precisely, 
for any character $\lambda$ of $G/H^j$, by definition, 
\begin{align}\label{mdibv}
\kappa^{j}_{\chi} \Big( \frac{1}{|G/H^j|} \sum\limits_{\bar{g} \in G/H^j} \lambda(\bar{g})
\bar{g} \Big) = \frac{1}{|G|} \sum\limits_{g \in G} \lambda \chi (g)g \in kG,
\end{align}
 where we also view $\lambda$ as a character of $G$. 
The following lemma describes  
$(G /H^j) \sharp k[W^i]$.

\begin{lemma}\label{Inject}
Fixing character $\chi$
and
its corresponding indecomposable idempotent $e_{\chi}$, there is a graded algebra injection 
$$
(G / H^j) \sharp k[W^{j}] \hookrightarrow G \sharp k[V] \cong G \sharp R.
$$		
\end{lemma}

\begin{proof}
By Proposition \ref{CanAb}, we know that $G$ can be viewed as a
group consisting of diagonal matrices with respect to its 
representation $V$. This diagonalization induces a basis $\{E_r \}_r$ of $V$. 
Since $H^j \subseteq G$, 
$H^j$ also can be viewed as 
group consisting of diagonal matrices with respect to its 
subrepresentation $W^j$ of $V$. Then there is a subset 
$\{E^j_r\}_r \subseteq \{E_r \}_r$ which is a base of $W^j$.  
Hence we get a splitting
$$V \cong W^j \oplus W'^j,$$
where $W'^j = 
\mathrm{Span}_{k} \{E_r \}_r 
\big\backslash \{E^j_r\}_r$. 
Then there is a natural projection 
$V \twoheadrightarrow W^j$, which is 
denoted by $q^j$. 
It induces an algebra injection 
$k[W^j] \hookrightarrow k[V] = R$, 
which is denoted by $q^j_{\#}$. 

Putting 
$\kappa^{j}_{\chi}$ and $q^j_{\#}$ together, 
we obtain an injection of graded vector space 
$$
\kappa^{j}_{\chi} \otimes q^j_{\#}: (G / H^j) \sharp k[W^{j}] \hookrightarrow G \sharp R. 
$$
Thus, to prove this lemma, it is left to prove that $\kappa^{j}_{\chi} 
\otimes q^j_{\#}$ is a graded algebra homomorphism. 

For $k(G / H^j)$, its indecomposable idempotents form 
a base
as a vector space. Thus, to prove this lemma, it suffices to show that 
$$
(\kappa^{j}_{\chi} \otimes q^j_{\#}(h_{\lambda_1}, f_1) )
 (\kappa^{j}_{\chi} \otimes q^j_{\#}(h_{\lambda_2}, f_2) ) = 
\kappa^{j}_{\chi} \otimes q^j_{\#}((h_{\lambda_1}, f_1)(h_{\lambda_2}, f_2) )
$$
where $f_1, f_2 \in k[W^j]$ are both monomials, 
and $h_{\lambda_1}, h_{\lambda_2} \in k[G/H^j]$ are indecomposable 
idempotents associated 
to $\lambda_1$ and $\lambda_2$ respectively. 

Now, let $R = k[V] = k[x_{1}, \cdots, x_n]$, 
where $x_r \in R$ corresponds to $E_r$ such that $x_r(E_r) = 1$ 
and 
$x_{r}(E_i) = 0$ when $i \neq r$. 
Now, $x_r$ is associated to a character, denoted by 
$\chi_{x_r}$, of $G$ such that $\chi_{x_r}(g) = 
g_{x_r} \in k^\ast$ for any $g \in G$, where $g_{x_r}$ is given by $g(x_r) = g_{x_r} x_r$.
More generally, 
for any monomial $f \in R$, we 
define a character $\chi_f$ of $G$ by the same way and let 
$\chi_{f}(g) = g_{f}$ for any $g \in G$. 
Hence we get that 
\begin{align}
(h_{\lambda_1}, f_1)(h_{\lambda_2}, f_2) & = (h_{\lambda_1}, f_1)
(\frac{1}{|G/H^j|} \sum\limits_{\bar{g} \in G/H^j} \lambda_2(\bar{g})\bar{g} , f_2 ) \nonumber \\ 
& = \frac{1}{|G/H^j|} \sum\limits_{\bar{g} \in G/H^j} 
(h_{\lambda_1} \lambda_2(\bar{g}) \bar{g}, \bar{g}^{-1}(f_1) f_2 ) \nonumber \\
& = \frac{1}{|G/H^j|} \sum\limits_{\bar{g} \in G/H^j} (h_{\lambda_1} 
\lambda_2(\bar{g}) \bar{g}, \chi^{-1}_{f_1}(\bar{g}) f_1 f_2 ) \nonumber \\
&= 
(h_{\lambda_1}, 1) \, \frac{1}{|G/H^j|} \sum\limits_{\bar{g} \in G/H^j} 
( \chi^{-1}_{f_1}\lambda_2(\bar{g}) \bar{g} ), f_1 f_2   \nonumber \\
& = (h_{\lambda_1} h_{\chi^{-1}_{f_1}\lambda_2}, \,  f_1 f_2). \label{ASD9864}
\end{align}
In the
above equalities, note that since $\chi_{f_1}(H^j) = \{1 \}$ by $f_1 \in k[W^j]$, $\chi_{f_1}$ 
can be viewed as a character of $G/H^j$.
Thus, by (\ref{ASD9864}) and
the definition of $\kappa^{j}_{\chi}$ 
(see (\ref{mdibv})), we have 
\begin{align}\label{jucbweib}
\kappa^{j}_{\chi} \otimes q^j_{\#} ( (h_{\lambda_1}, f_1)(h_{\lambda_2}, f_2) ) 
& = \kappa^{j}_{\chi} \otimes q^j_{\#}
(h_{\lambda_1} h_{\chi^{-1}_{f_1}\lambda_2}, \,  f_1 f_2)\nonumber\\
& = ( \kappa^{j}_{\chi}(h_{ \lambda_1} 
h_{\chi^{-1}_{f_1}\lambda_2}), q^j_{\#}(f_1 f_2) ) \nonumber \\ 
& = ( h_{ \chi \lambda_1} h_{\chi \chi^{-1}_{f_1}\lambda_2}, q^j_{\#}(f_1 f_2) ). 
\end{align}

Meanwhile, by the definition of $\kappa^{j}_{\chi}$ we also get that 
$\kappa^{j}_{\chi} \otimes q^j_{\#}(h_{\lambda_1},  f_1 ) = 
(h_{ \chi \lambda_1}, q^j_{\#}(f_1))$ and 
$\kappa^{j}_{\chi} \otimes q^j_{\#}(h_{\lambda_2},  f_2 ) = 
(h_{ \chi \lambda_2}, q^j_{\#}(f_2))$. 
By the same computation as
(\ref{ASD9864}), we have 
\begin{align}\label{jucbweibdd}
(h_{ \chi \lambda_1}, q^j_{\#}(f_1)) (h_{ \chi \lambda_2}, q^j_{\#}(f_2)) 
= (h_{ \chi \lambda_1} h_{\chi \chi^{-1}_{f_1}\lambda_2} , q^j_{\#}(f_1) q^j_{\#}(f_2)). 
\end{align}
Here, we identify $\chi_{f_1}$ with $\chi_{q^j_{\#}(f_1)}$ since $q^j_{\#}$ is an injection. 
Since $q^j_{\#}(f_1) q^j_{\#}(f_2) = q^j_{\#}(f_1 f_2)$, 
by 
(\ref{jucbweib}) and (\ref{jucbweibdd}), we get that 
$$
(\kappa^{j}_{\chi} \otimes q^j_{\#}(h_{\lambda_1}, f_1) ) (\kappa^{j}_{\chi} 
\otimes q^j_{\#}(h_{\lambda_2}, f_2)) = 
\kappa^{j}_{\chi} \otimes q^j_{\#}((h_{\lambda_1}, f_1)(h_{\lambda_2}, f_2) ).
$$
This proves the lemma. 
\end{proof}

Denote by $\bar{\kappa}^{j}_{\chi}$  
the injection in the above lemma. 
It means that the 
$\kappa^{j}_{\chi}$ can 
be extended 
from $k(G / H^j)$ to
$(G / H^j) \sharp k[W^{j}]$.
In the meantime, we have the following. 

\begin{lemma}\label{Compos}
With the above settings, there is a graded algebra surjection
$$
	\bar{\varsigma}^{j}_{\chi}: G \sharp R  \twoheadrightarrow   (G / H^j) \sharp k[W^{j}]
$$
such that $\bar\varsigma^{j}_{\chi} \circ \bar\kappa^{j}_{\chi} 
= \mathrm{Id}$ in $(G / H^j) \sharp k[W^{j}]$. 
\end{lemma}

\begin{proof}
We first show the existence
of the surjection as graded vector 
spaces.
In fact, firstly by Lemma \ref{Subalge}, there is an algebra homomorphism 
$$
(\kappa^{j}_{\chi})^{-1}: \bigoplus_{i} ke_{i}^j  \xrightarrow{\sim} k(G / H^j).
$$
Secondly, 
by Proposition \ref{ReChaId}(2), 
there is
a canonical algebra projection 
$$
kG \twoheadrightarrow \bigoplus_{i} ke_{i}^j,
$$
which we denote by $\Delta^{j}_{\chi}$. Here, 
$\Delta^{j}_{\chi}(e_{i}^j) = e_{i}^j$ and $\Delta^{j}_{\chi}(h) = 0$ 
for any 
indecomposable idempotent 
$h \notin \{ e_{i}^j\}_i$. 
Let $$\varsigma^{j}_{\chi} :=  
(\kappa^{j}_{\chi})^{-1} \circ 
\Delta^{j}_{\chi}: k 
G\twoheadrightarrow
k(G/H^j).$$ 
Note that $(\kappa^{j}_{\chi})^{-1}
(h_{\chi'}) = h_{\chi' \chi^{-1}}$ 
for any 
$h_{\chi'} \in \{ e_i^j \}_i$.
Thirdly, 
we have a natural graded algebra  
projection 
$$
\iota^j_{\#}:
R = k[V] \twoheadrightarrow k[W^j],$$
which is induced by injection $\iota^j: W^j \hookrightarrow V$. 
Thus, we get a surjection of graded vector spaces 
$$
\varsigma^{j}_{\chi} \otimes \iota^j_{\#}: G \sharp R \cong G \sharp k[V]  
\twoheadrightarrow   (G / H^j) \sharp k[W^{j}],
$$
which we denote by $\bar{\varsigma}^{j}_{\chi}$.  

Next, we show that $\bar{\varsigma}^{j}_{\chi}$ is an algebra homomorphism. 
To this end,
let $(h_{\chi_1}, f_1), (h_{\chi_2}, f_2) \in G \sharp R$. 
Without loss of generality, assume 
that $f_1, f_2$ are both monomials. 
To prove that $\bar{\varsigma}^{j}_{\chi}$ is an algebra homomorphism, it is enough to show that 
\begin{align}\label{dncjw}
	(\bar{\varsigma}^{j}_{\chi}(h_{\chi_1}, f_1) 
    )(\bar{\varsigma}^{j}_{\chi}(h_{\chi_2}, f_2)) 
    = \bar{\varsigma}^{j}_{\chi} ((h_{\chi_1}, f_1) (h_{\chi_2}, f_2)).
\end{align}
We prove it by considering
the following two cases.

\vspace{2mm}
\noindent{\bf Case 1:} 
at least one of $\chi_1$ and $\chi_2$ is not in 
$\{\chi_i^j\}_i$. In this case, 
at least one of 
$\Delta^{j}_{\chi}(\chi_1), \Delta^{j}_{\chi}(\chi_2)$ is zero. It implies that 
$\bar{\varsigma}^{j}_{\chi}(h_{\chi_1}, f_1)\bar{\varsigma}^{j}_{\chi}(h_{\chi_2}, f_2) = 0  
$ by definitions of $\bar{\varsigma}^{j}_{\chi}$ and 
$\varsigma^{j}_{\chi} = (\kappa^{j}_{\chi})^{-1} \circ \Delta^{j}_{\chi}$. 
On the other hand, we claim that
$\bar{\varsigma}^{j}_{\chi} ((h_{\chi_1}, f_1) (h_{\chi_2}, f_2)) 
= \bar{\varsigma}^{j}_{\chi} (h_{\chi_1} h_{\chi^{-1}_{f_1} \chi_2}, f_1 f_2) =0$,
due to the following reason:
\begin{enumerate}
\item[(1)] if 
$\chi_1 \notin \{\chi_i^j\}_i$, then 
$\Delta^{j}_{\chi}(h_{\chi_1}) = 0$, which implies that 
$$\Delta^{j}_{\chi}(h_{\chi_1} h_{\chi^{-1}_{f_1} \chi_2}) = \Delta^{j}_{\chi}(h_{\chi_1}) 
\Delta^{j}_{\chi}(h_{\chi^{-1}_{f_1} \chi_2}) = 0;$$
	
\item[(2)] if $\chi_1 \neq \chi^{-1}_{f_1} \chi_2$, then 
$h_{\chi_1} \neq h_{\chi^{-1}_{f_1} \chi_2}$, which implies that  
$h_{\chi_1} h_{\chi^{-1}_{f_1} \chi_2} = 0$; 

\item[(3)] if $\chi_1 \in \{\chi_i^j\}_i$, $\chi_1 = \chi^{-1}_{f_1} \chi_2$ but $\chi_2 
\notin \{\chi_i^j\}_i$, then $\chi_{f_1}|_{H^j}$ is nontrivial, which implies that the action of 
$H^j$ on $f_1$ is nontrivial. Hence, we know that $f_1 \notin k[W^j]$. 
It follows that $\iota^j_{\#}(f_1) = 0$. Then $\iota^j_{\#}(f_1 f_2) = 0$. 
\end{enumerate}
We thus have proved
(\ref{dncjw}) if at least 
one of $\chi_1, \chi_2$ is not in $\{\chi_i^j\}_i$.

\vspace{2mm}
\noindent{\bf Case 2:}
both
$\chi_1,\chi_2 \in \{x_i^j\}_i$.
Let us first look at the right-hand
side of \eqref{dncjw}.
By the similar computation 
to (\ref{ASD9864}), 
we have 
$$
(h_{\chi_1}, f_1) (h_{\chi_2}, f_2) = (h_{\chi_1} h_{\chi^{-1}_{f_1} \chi_2}, f_1 f_2). 
$$
Since $\chi_1 \in \{x_i^j\}_i$, we have $\Delta^{j}_{\chi}(h_{\chi_1}) = h_{\chi_1}$. 
Then the right-hand side
\begin{align}\label{gjkidb8520}
\bar{\varsigma}^{j}_{\chi}(h_{\chi_1} h_{\chi^{-1}_{f_1} \chi_2}, f_1 f_2) 
= ( \varsigma^{j}_{\chi}(h_{\chi_1} h_{\chi^{-1}_{f_1} \chi_2}), \iota^j_{\#}(f_1 f_2) ) = 
( h_{\chi_1 \chi^{-1}} \varsigma^{j}_{\chi}(h_{\chi^{-1}_{f_1} \chi_2} ),\, 
\iota^j_{\#}(f_1 f_2) ).
\end{align}
We next look at the left-hand
side of \eqref{dncjw}. We first have 
\begin{align}
\bar{\varsigma}^{j}_{\chi}(h_{\chi_1}, f_1)\bar{\varsigma}^{j}_{\chi}(h_{\chi_2}, f_2) 
& = ( \varsigma^{j}_{\chi}(h_{\chi_1}), \iota^j_{\#}(f_1) ) 
( \varsigma^{j}_{\chi}(h_{\chi_2}), \iota^j_{\#}(f_2) ) \nonumber \\ 
& = (h_{\chi_1 \chi^{-1}}, \iota^j_{\#}(f_1) ) ( h_{\chi_2 \chi^{-1}}, 
\iota^j_{\#}(f_2) ) \nonumber \\
& = ( h_{\chi_1 \chi^{-1}} h_{\chi_2 \chi^{-1}\chi_{\iota^j_{\#}(f_1)}^{-1}}, 
\iota^j_{\#}(f_1) \iota^j_{\#}(f_2) ) \nonumber \\
& = (h_{\chi_1 \chi^{-1}} h_{\chi_2 \chi^{-1}\chi_{\iota^j_{\#}(f_1)}^{-1}}, \iota^j_{\#}(f_1 f_2) ).  
\label{BJGC61037}
\end{align}
If
$f_1 \in \mathrm{Ker}
(\iota^j_{\#})$, then by 
(\ref{gjkidb8520}) and 
(\ref{BJGC61037}) we have
\begin{equation}\label{eq91}
 \bar{\varsigma}^{j}_{\chi}(h_{\chi_1}, f_1)\bar{\varsigma}^{j}_{\chi}(h_{\chi_2}, f_2) 
 =0
 = \bar{\varsigma}^{j}_{\chi} ((h_{\chi_1}, f_1) (h_{\chi_2}, f_2)).
\end{equation}
Otherwise if $f_1 \notin \mathrm{Ker}(\iota^j_{\#})$, 
by the definitions of $\iota^j_{\#}$ and $q^{j}_{\#}$, we know that 
$f_1 \in k[W^j]$. It implies that 
$f_1 = \iota^j_{\#}(f_1)$. Hence, 
$\chi_{f_1} = \chi_{\iota^j_{\#}(f_1)}$. Then 
\begin{align}\label{hjsbjk}
h_{\chi_2 \chi^{-1}\chi_{\iota^j_{\#}(f_1)}^{-1}} =  h_{\chi_2 \chi^{-1} \chi^{-1}_{f_1}}.
\end{align}
Moreover, since $f_1 \in k[W^j]$, 
$\chi_{f_1}(H^j) = \{1 \}$. It suggest that 
$\chi_{2}\chi_{f_1} \in \{ \chi_i^j \}_i$. 
Thus, we have 
\begin{align}\label{hjsbjk1}
\varsigma^{j}_{\chi}(h_{\chi^{-1}_{f_1} \chi_2} ) 
=  (\kappa^{j}_{\chi})^{-1}\circ \Delta^{j}_{\chi}(h_{\chi^{-1}_{f_1} \chi_2} ) 
= (\kappa^{j}_{\chi})^{-1}(h_{\chi^{-1}_{f_1} \chi_2} ) = h_{\chi^{-1}_{f_1} \chi_2 \chi^{-1}}.
\end{align}
Comparing (\ref{gjkidb8520}) and (\ref{BJGC61037}), we again get that  
\begin{equation}\label{eq92}
 \bar{\varsigma}^{j}_{\chi}(h_{\chi_1}, f_1)\bar{\varsigma}^{j}_{\chi}(h_{\chi_2}, f_2) 
 = \bar{\varsigma}^{j}_{\chi} ((h_{\chi_1}, f_1) (h_{\chi_2}, f_2)) 
\end{equation}
by (\ref{hjsbjk}) and (\ref{hjsbjk1}). 
Combining \eqref{eq91}
and \eqref{eq92} we get
\eqref{dncjw} in this case.

Finally, we show that 
$\bar\varsigma^{j}_{\chi} \circ \bar\kappa^{j}_{\chi} 
= \mathrm{Id}$.
In fact,
by the definitions of above homomorphisms, 
it is direct to see that 
\begin{align*}
\bar\varsigma^{j}_{\chi} \circ \bar\kappa^{j}_{\chi} & 
=  ( \varsigma^{j}_{\chi} \otimes \iota^j_{\#}) \circ (\kappa^{j}_{\chi} \otimes q^{j}_{\#}) \\
& = ((\varsigma^{j}_{\chi} \circ \kappa^{j}_{\chi}) \otimes (\iota^j_{\#} \circ  q^{j}_{\#}) ) \\ 
& = \mathrm{Id}  
\end{align*}
on $(G / H^j) \sharp k[W^{j}]$. 
We thus have completed the proof. 
\end{proof}

\subsection{The algebra $\Lambda/ \Lambda e\Lambda$}

From now on, we prove
Theorem \ref{thm:smoothcatcpt2}.
As sketched in \S\ref{subsect:strategy},
our first step is to give
a surjective map from
$\Lambda/\Lambda e\Lambda$.
In fact, we have more.

\begin{proposition}\label{Modstru}
Fix a pair $(\chi, H^j)$ such that 
$\chi|_{H^j}$ is nontrivial. Then there is a graded algebra surjection 	
$$
\eta^{j}_{\chi}: \Lambda / \Lambda e \Lambda \twoheadrightarrow (G / H^j) \sharp k[W^{j}]
$$		
and a graded algebra injection 
$$
\lambda^{j}_{\chi}: (G / H^j) \sharp k[W^{j}] \hookrightarrow \Lambda / \Lambda e \Lambda
$$	
such that $\eta^{j}_{\chi} \circ \lambda^{j}_{\chi}  = \mathrm{Id}$. 
\end{proposition}

\begin{proof}
We only need to construct
$\eta^{j}_{\chi}$ and 
$\lambda^{j}_{\chi}$ and show
that their composition is
the identity.

First, let 
$\lambda^{j}_{\chi}$ be the following composition: 	
$$
(G / H^j) \sharp k[W^{j}] \xrightarrow{\bar\kappa^{j}_{\chi}} 
\Lambda \twoheadrightarrow \Lambda / \Lambda e \Lambda, 
$$
where 
$\bar\kappa^{j}_{\chi}$
is given by Lemma \ref{Inject},
and
$\Lambda \twoheadrightarrow \Lambda / \Lambda e \Lambda$ is the natural algebra projection. 

Next, we introduce $\eta_\chi^j$.
Let $\chi_0$ be the trivial character of $G$. Since $\chi$ is not trivial when restricting 
to $H^j$, $\chi_0 \notin \{ \chi_i^j \}_i$. 
It implies that $\Delta^{j}_{\chi}(e) = 0$.
Since
$$
\bar\varsigma^{j}_{\chi}(e) = \bar\varsigma^{j}_{\chi}(h_{\chi_0} \otimes 1) 
= ( (\kappa^{j}_{\chi})^{-1} \circ \Delta^{j}_{\chi}(h_{\chi_0})  \otimes \iota^j_{\#}(1) ) = 0,  
$$
where $e$ is viewed as an idempotent of 
$\Lambda$ (note $h_{\chi_0} = e$), 
it follows that $\bar\varsigma^{j}_{\chi}(\Lambda e \Lambda) = \{ 0\}$. 
Then
by Lemma 
\ref{Compos},
$\bar\varsigma^{j}_{\chi}$ induces a graded algebra surjection 
$$
\Lambda / \Lambda e \Lambda \twoheadrightarrow (G / H^j) \sharp k[W^{j}],
$$
which is our
$\eta^{j}_{\chi}$. 

Finally, 
since $\bar\varsigma^{j}_{\chi} \circ \bar\kappa^{j}_{\chi}
= \mathrm{Id} $
on $(G / H^j) \sharp k[W^{j}]$, we
have $\eta^{j}_{\chi} \circ 
\lambda^{j}_{\chi}  = \mathrm{Id}$ 
on $(G / H^j) \sharp k[W^{j}]$ by the
construction of the above graded algebra
homomorphisms. 
\end{proof}

From the constructions of $\eta^{j}_{\chi}$ and $\lambda^{j}_{\chi}$, we have the following.
 
\begin{corollary}\label{DiffIde2}
With the above settings, let $e'$ and $e''$ be two indecomposable idempotents of $\Lambda$. Suppose that 
$(\chi', H^j) \sim (\chi'', H^j)$, where $\chi'$ and $\chi''$ correspond to $e'$ and $e''$ respectively.
Then 
$\eta^{j}_{\chi'} = \eta^{j}_{\chi''}$ and $\lambda^{j}_{\chi'} = \lambda^{j}_{\chi''}$.
\end{corollary}

From Proposition \ref{Modstru}, 
for any pair $(\chi, G^i)$, we get
an algebra surjection 
$\eta^{i}_{\chi}$ by replacing $H^j$ with 
$G^i$. Then by Corollary
\ref{DiffIde2}
we have the following well-defined 
graded algebra homomorphism 
$$
\varpi_0 : =  \bigoplus_{(\chi, G^i) \in \tilde{G}_0} \eta^{i}_{\chi}: 
\Lambda / \Lambda e \Lambda \rightarrow \bigoplus_{(\chi, G^i) \in \tilde{G}_0} \Lambda_i,
$$ 
where $\Lambda_i := (G/G^i) \sharp k[V^i]$. 
As we stated before, we 
wanted to prove $\varpi_0$
is an injection of algebras, but
this is not the case; however, on
the other hand, $\mathrm{ker}\varpi_0$
only consists of nilpotent elements.
We introduce the following
algebra $\Lambda/\sqrt{\Lambda e\Lambda}$, which
removes these nilpotent elements.
It plays an important
role in the proof of Theorem \ref{thm:smoothcatcpt2} as follows:
on the one hand, the algebra $\Lambda / \sqrt{\Lambda e\Lambda}$ can be decomposed 
into AS regular algebras with respect to pairs in $\tilde{G}_0$ (see \eqref{jxkdbn}), 
and then it admits a graded projective $(\Lambda / \sqrt{\Lambda e\Lambda})^e$-module 
resolution (see \ref{gnkbvj}); on the another hand, any classical generator of  
$D^{b}_{\mathrm{tail}(\Lambda/ \Lambda e \Lambda)}(\mathrm{tail}(\Lambda))$ 
gives a classical generator of  
$D^{b}_{\mathrm{tail}(\Lambda/ 
\sqrt{ \Lambda e \Lambda})}(\mathrm{tail}(\Lambda))$ (see Lemma \ref{nribe} below).

\subsection{The algebra $\Lambda/ \sqrt{\Lambda e\Lambda}$}

Consider the following two-sided ideal of $\Lambda$. 

\begin{definition}\label{bdfbv11267}
Let
$$
\sqrt{\Lambda e\Lambda} := \mathrm{Span}_{k}\Big\{e' \otimes f 
\in \Lambda \big\vert \exists m \in \mathbb{N}
\mbox{ such that }
e' \otimes f^m \in \Lambda e\Lambda \Big\}. 
$$
\end{definition}

Consider $\bar{\varsigma}_{\chi}^{i} ( \sqrt{\Lambda e\Lambda} )$ 
for any pair $(\chi, G^i) \in \tilde{G}_0$. 
By the definitions of $\sqrt{\Lambda e\Lambda}$ and $\bar{\varsigma}_{\chi}^{i}$, we know that 
for any $e' \otimes f \in \sqrt{\Lambda e\Lambda}$, 
$\bar{\varsigma}_{\chi}^{i}(e' \otimes f^m ) = 0$ for some $m \in \mathbb{N}$ 
(see the proof of Proposition \ref{Modstru}),
and
then either $\varsigma_{\chi}^{i}(e') = 0$ or 
$ p_{\#}^{i}(f^{m}) = 0$. Since $k[V^i]$ is reduced, 
$(p_{\#}^{i}(f))^m = p_{\#}^{i}(f^{m}) = 0$ implies that $ p_{\#}^{i}(f) = 0$ and thus 
$\bar{\varsigma}_{\chi}^{i}(e'' \otimes f ) = 0$. 
Thus for any pair $(\chi, G^i) \in \tilde{G}_0$, 
from $\eta_{\chi}^{i}$
we also obtain a graded algebra surjection  
$$
\bar{\eta}_{\chi}^{i}: \Lambda / \sqrt{\Lambda e\Lambda} \twoheadrightarrow \Lambda_i,
$$
and hence the following natural graded algebra homomorphism 
\begin{align}\label{jxkdbn}
\bar{\varpi}_0 : =  \bigoplus_{(\chi, G^i) \in \tilde{G}_0} \bar{\eta}^{i}_{\chi}: 
\Lambda / \sqrt{\Lambda e \Lambda} \rightarrow \bigoplus_{(\chi, G^i) \in \tilde{G}_0} \Lambda_i.
\end{align}
which may also be viewed as
being induced from $\varpi_0$.

\begin{lemma}\label{NNned1}
$\bar{\varpi}_0$ is injective. 
\end{lemma}

\begin{proof}
This proof is a bit involved and
consists of several steps.

\vspace{2mm}

\noindent{\bf Step 1. We
reduce the
proof of the lemma
to the proof of the following (\ref{hxowc453}).}

In fact, by the definition 
(\ref{jxkdbn}), notice that to 
prove the lemma,
it suffices to show that 
$$
\mathrm{Ker}(\bar{\varpi}_0) \cong \bigcap\limits_{(\chi, G^i) \in \tilde{G}_0}
 \mathrm{Ker}(\bar{\eta}^{i}_{\chi})	
$$	
is trivial on $\Lambda / \sqrt{\Lambda e \Lambda}$.

Since $\eta^{i}_{\chi}$ is also a graded $\Lambda$-module morphism,
$\mathrm{Ker}(\eta^{i}_{\chi})$ is a graded $\Lambda$-module. Thus, it suffices to show that for any 
indecomposable idempotent $u$ of $\Lambda$, 
\begin{align}\label{qsxbd83}	
\Big(\bigcap\limits_{(\chi, G^i) \in \tilde{G}_0} \mathrm{Ker}(\eta^{i}_{\chi}) \Big)\otimes_{\Lambda} \Lambda u 
& \cong  \bigcap\limits_{(\chi, G^i) \in \tilde{G}_0} ( \mathrm{Ker}(\eta^{i}_{\chi}) 
\otimes_{\Lambda} \Lambda u ) \nonumber \\
 & \cong \bigcap\limits_{(\chi, G^i) \in \tilde{G}_0} \mathrm{Ker} (\eta^{i}_{\chi} 
 \otimes_{\Lambda} \Lambda u ) 
\end{align}
is trivial as graded $R$-modules. 

Next, $\sqrt{\Lambda e \Lambda} \otimes_{\Lambda} \Lambda u \cong (\sqrt{\Lambda e \Lambda} )u
\subseteq \Lambda u \cong R$ is a reduced ideal, denoted by $I_u$, of $R$.  
Let $\chi_u$ be the character of $G$ corresponding to $u$. 
Then we have graded algebra homomorphism $$\eta^{i}_{\chi} \otimes_{\Lambda} \Lambda u : 
R/I_u \cong \Lambda u \big/ (\sqrt{\Lambda e \Lambda} )u \twoheadrightarrow  \Lambda_i \otimes_{\Lambda} \Lambda u \cong k[V^i].$$
It is nontrivial if and only 
if $(\chi_u, G^i) \sim (\chi, G^i)$ in $\tilde{G}_0$. 
Now assume
that $(\chi_u, G^i) \sim (\chi, G^i)$. 
Let $\iota_{\#}^i : R \twoheadrightarrow k[V^i]$ be the canonical projection and $I^i$ be its kernel. 
Here, we note that
$\iota_{\#}^i$ induces $\eta^{i}_{\chi} 
\otimes_{\Lambda} \Lambda u : R/ I_u \twoheadrightarrow R/I^i$. 

Now, by the construction of $\tilde{G}_0$, to prove (\ref{qsxbd83}) is trivial, it is enough to show that 
\begin{align}\label{2736sdd}
\bigcap\limits_{(\chi_0, G^i) \not\sim (\chi_u, G^i)} I^i \subseteq I_u.
\end{align}
To this end, assume 
$f \in \bigcap\limits_{(\chi_0, G^i) 
\not\sim (\chi_u, G^i)} I^i $ is 
a monomial.
We need to show $f \in I_u$. 
In fact,
by the definition of 
$\sqrt{\Lambda e \Lambda}$ 
(see Definition \ref{bdfbv11267}), 
we know that 
\begin{align}\label{ahjd 710}
\sqrt{\Lambda e \Lambda} \otimes_{\Lambda} \Lambda u \cong 
\mathrm{Span}_{k}\Big\{ e' \otimes f \in 
\Lambda u \big\vert \exists m \in 
\mathbb{N}\mbox{ such that }
e' \otimes f^m \in \Lambda e \Lambda  
\Big\}. 
\end{align}
Thus to prove $f \in I_u$, it is enough to prove that 
\begin{align}\label{hxowc453}
(1 \otimes f^l)u \in \Lambda e \Lambda u,
\end{align}
for some $l \in \mathbb{N}$.

\vspace{2mm}
\noindent{\bf Step 2. We reduce the proof of
(\ref{hxowc453}) to the proof of the following (\ref{ncivfbg}).}

By the definition of $\iota_{\#}^i$, we know that 
$I^i$ is generated by elements in $k[V'^{i}]^{+}$, where $V'^{i}$ is the
kernel of the canonical projection $q^i: V \rightarrow V^i$ 
(see the definition in the proof of Lemma \ref{Inject}) and 
$k[V'^{i}]^{+}$ is given by the natural splitting 
$k[V'^{i}] \cong k[V'^{i}]^{+} \oplus k$. 
Without loss of generality, by the assumption of $f$, we let 
$$
f = \bigodot\limits_{(\chi_0, G^i) \not\sim (\chi_u, G^i)} f_i,
$$
where $\bigodot$ means 
the product of $R$, and $f_i \in I^i$ is a variable. 

To this end, consider the subgroup
$\langle\chi_{f_i}\rangle$, which 
is generated by $\chi_{f_i}$, in 
the dual group
$(G^i)^{\vee}$ of $G^i$. The following equality is straightforward: 
\begin{align}\label{BUXGECG3}
\#(\langle\chi_{f_i}\rangle) \#(G_{f_i}) = \#(G^i),
\end{align}
where $\#(-)$ is
the number of elements in group and
$G_{f_i} := \{g \in G^i \mid g(f_i) = f_i \}$. 
It follows that 
\begin{align}\label{BUXGECG322}
\#(\langle\chi_{f_i}\rangle)  = \#(G^i / G_{f_i}),
\end{align}
and hence that $\langle\chi_{f_i}\rangle 
=  (G^i / G_{f_i})^{\vee}$ in $(G^i)^{\vee}$.  

Now, from the definition of the
pair $(\chi_u, G^i) \in \tilde{G}_0$, we know
that there is no other $K^l$ such that
$K^l \subseteq G^i$ and $\chi_u \mid_{K^l}$ 
is nontrivial.
Meanwhile, we have $G_{f_i} \subsetneq G^i$ and $G_{f_i} \in \{K^l\}_l$ by the fact that $f_i$ is a variable.
Then we have that $\chi_u \mid_{G_{f_i}}$ is trivial. 
Hence, $\chi_u \in (G^i/G_{f_i})^{\vee}$.
We get that $\chi_{0} \chi_{f_i}^{m} = \chi_{f_i}^{m} 
= \chi_u$ in $(G^i/G_{f_i})^{\vee}$ for some $m \in \mathbb{N}$ by $\langle\chi_{f_i}\rangle 
=  (G^i / G_{f_i})^{\vee}$. 
Then we obtain that
$\chi_{0} \chi_{f_i}^{m} = \chi_u$ in $(G^i)^\vee$ by $(G^i/G_{f_i})^{\vee} \subseteq (G^i)^\vee$, and hence
that 
\begin{align}\label{nxje987}
(\chi_u \chi_{f_i}^{m'}, G^i) \sim (\chi_0, G^i)
\end{align}
in $\tilde{G}$ 
for some $m' \in \mathbb{N}$ such that $\chi_{f_i}^{-m} = \chi_{f_i}^{m'}$.
Let $e^i$ be the indecomposable idempotent such that
$ (1 \otimes f_i^{m}) u = e^i (1 \otimes f_i^{m})u \in \Lambda u$. 
By 
\begin{align}\label{jvxw56} 
(1 \otimes f_i^{m})u &= (1 \otimes f_i^{m}) 
(\frac{1}{|G|} \sum\limits_{g \in G} \chi_u(g) g \otimes 1)	 \nonumber \\
&= \frac{1}{|G|} \sum\limits_{g \in G} 
( \chi_u(g) g  \otimes g^{-1}(f_i^{m})) \nonumber \\
& =\displaystyle
\frac{1}{|G|} \sum\limits_{g \in G} 
( \chi_{f_i}^{-m}(g) \chi_u(g) g  \otimes f_i^{m}) \nonumber \\ 
& =(\frac{1}{|G|} \sum\limits_{g \in G} 
\chi^{m'}_{f_i}\chi_u(g)g  \otimes 1 )  (1 \otimes f_i^{m}),
\end{align}
we get that $e^i = \displaystyle
\frac{1}{|G|} 
\sum\limits_{g \in G} 
\chi^{m'}_{f_i}\chi_u(g)g  \otimes 1$. 
Next, by 
(\ref{nxje987}), 
we have that 
$$
\varsigma^{i}_{\chi_0}(e^i)  = \varsigma^{i}_{\chi_0} ( \frac{1}{|G|} 
\sum\limits_{g \in G} \chi^{m'}_{f_i}\chi_u(g)g  \otimes 1 ) \neq 0,
$$ 
where $\varsigma^i_{\chi_0}$ corresponds 
to the pair $(\chi_0, G^i)$. 
By the fact that $\varsigma^{i}_{\chi_0}
\circ \kappa^{i}_{\chi_0} = \mathrm{Id}$, and the definitions of 
$\varsigma^{i}_{\chi_0}$ and $\kappa^{i}_{\chi_0}$(see the proof of Lemma \ref{Compos}), 
it follows that $e^i$ is in the image of $\kappa^i_{\chi_0}$, where 
$\kappa^i_{\chi_0}$ corresponds to the pair $(\chi_0, G^i)$.  
Thus, $e^i$ is in the image of $\bar\kappa^i_{\chi_0}$.

In what follows, we replace $G$, $V$ 
and $u$ by $G/G^i$, $V^i$ and $e^i$ 
respectively. Let $f'$ be a factor of $f$ 
satisfying 
that it is the product of 
all variables like $f_r \in \{ f_i \}
$ 
such that $G^i \subsetneq G^r $, 
where $G^r$ corresponds to $f_r$. 
We can write $f'$ as
$$
\bigodot\limits_{f_r \in k[V^i]} f_r. 
$$
Then we have 
$$ f' \in \bigcap\limits_{(\chi_0, G^r/G^i) \not\sim (\chi_{e^i}, G^r/G^i)}  I^r,
$$ 
where $\chi_{e^i}$ is the character corresponding to the
idempotent $e^i$. Now, we also replace $f$ by $f'$. 
Note that since the action of $G^i$ on each $f_r$ is trivial, 
the action of $G^i$ on $f'$ is also trivial.
It implies that for any idempotent $u', u''$ in $\mathrm{Im}(\bar{\kappa}^i_{\chi_0})$, 
$u'(1 \otimes f')u''$ is also in $\mathrm{Im}(\bar{\kappa}^i_{\chi_0})$.

Since $(e^i \otimes f_i^{m}) \in \Lambda u $, to prove that 
$(1 \otimes f^l)u \in \Lambda e \Lambda u$
(see \ref{hxowc453}) for some 
$l \in \mathbb{N}$, 
it suffices to show that
$(1 \otimes f'^{l'})e^i \in \Lambda e \Lambda e^i$, for some
$l' \in \mathbb{N}$. 
Since $e, e^i, (1 \otimes f'^{l'})e^i \in 
\mathrm{Im}(\bar{\kappa}^i_{\chi_0})$, it is left to show that 
\begin{align}\label{ncivfbg}
	(1 \otimes f'^{l'})e^i \in \Lambda_i e \Lambda_i e^{i},
\end{align}
 where 
we identity $\Lambda_i$ with the image 
of injection $\bar{\kappa}^i_{\chi_0}$. 
In other words,
the proof of the lemma for $\Lambda$ to 
is reduced to the one for $\Lambda_i$. 

\vspace{2mm}
\noindent{\bf Step 3. We prove
(\ref{ncivfbg}) by reducing it to the case that $\tilde{G}_0$ is empty.}

We repeatedly apply the above procedure. 
Note that the number of equivalence classes
of $\{\chi, G^r\}_r$ in $\tilde{G}_0$ 
such that $G^i \subsetneq G^r$
is less than the number of equivalence
classes of $\{\chi, G^i\}_i$ in 
$\tilde{G}_0$, and thus
by reduction, the proof of this lemma for $\Lambda_i$ 
is reduced to the case 
where $\Lambda$ satisfies that
the set $\tilde{G}_0$ of equivalence classes of $\{\chi, G^i\}_i$ is empty. 

In this case, we have $\chi(K^l) = 1$ for any character $\chi$ and $K^l$. Since $G \in \{K^l\}_l$, $G$ is trivial. Hence 
$(1 \otimes f'^{l'})e^i \in \Lambda_i e \Lambda_i e^{i}$(see \ref{ncivfbg}) holds immediately. 
The proof is now complete. 
%
%
%
\end{proof}

This lemma will be used
in the next subsection.

\subsection{The resolution of $\Lambda/ \sqrt{\Lambda e\Lambda}$}

Consider the graded
$(\Lambda/\sqrt{\Lambda e \Lambda})^e$-module $\mathrm{Im}(\bar{\varpi}_0)$ 
and the following graded $(\bigoplus\limits_{(\chi, G^i)} \Lambda_{i})$-module homomorphisms. 
Let $(\lambda, G^{i_1})$ and $(\lambda, G^{i_2})$ be two pairs in $\tilde{G}_0$.
Let $V^{i_1, i_2} = V^{i_1} \cap V^{i_2}$ where 
$V^{i_1}, V^{i_2}$ are the invariant subspaces of $G^{i_1}$ and $G^{i_2}$ 
respectively. Moreover, let $K^{i_1, i_2} \subseteq G$ be the subgroup such 
that $K^{i_1, i_2} \subseteq \{ K^l \}_l$ and $V^{i_1, i_2}$ 
is its invariant subspace. 
There are two natural graded 
algebra projections 
$$\theta^{i_1} : k[V^{i_1}] 
\twoheadrightarrow k[V^{i_1, i_2}]\quad\mbox{and}\quad
\theta^{i_2} : k[V^{i_2}] \twoheadrightarrow k[V^{i_1, 
i_2}].$$
Similarly to the construction 
of $\bar\varsigma^{j}_{\chi}$, we have graded 
algebra surjections 
$$\bar\theta^{i_1}_{\lambda}: G^{i_1} \sharp k[V^{i_1}] 
\twoheadrightarrow G / K^{i_1, i_2} \sharp k[V^{i_1, i_2}]
\quad\mbox{and}\quad
\bar\theta^{i_2}_{\lambda}:  G^{i_2} \sharp k[V^{i_2}] 
\twoheadrightarrow G / K^{i_1, i_2} \sharp k[V^{i_1, 
i_2}].$$
They induce a graded 
$(\Lambda / \sqrt{\Lambda e \Lambda})^e$-module
homomorphism 
$$
\bar\theta^{i_1, i_2}_\lambda :  G^{i_1} \sharp k[V^{i_1}] 
\oplus  G^{i_2} \sharp k[V^{i_2}] 
\xrightarrow{\bar\theta^{i_1}_{\lambda} 
\oplus (-\bar\theta^{i_2}_{\lambda})} 
G / K^{i_1, i_2} \sharp k[V^{i_1, i_2}],
$$
where the graded
$(\Lambda / \sqrt{\Lambda e \Lambda})^e$-module structures
of $$(G /G^{i_1} \sharp k[V^{i_1}]) 
\oplus(G /G^{i_2} \sharp k[V^{i_2}])\quad
\mbox{and}
\quad
G /K^{i_1, i_2} \sharp k[V^{i_1, i_2}]$$ are given by 
$\bar{\eta}^{i_1}_{\lambda} \oplus 
\bar{\eta}^{i_2}_{\lambda}$ and $ \bar{\eta}^{i_1, 
i_2}_{\lambda}$, 
which are associated to
the
subgroups $G^{i_1}$, $G^{i_2}$ and $K^{i_1, i_2}$ respectively. 

Let $(\lambda_1, G^{i_1}), (\lambda_2, G^{i_2}) \in \tilde{G}_0$ be 
two pairs such that there exists a character 
$\lambda$ such that $(\lambda, G^{i_1}) \sim (\lambda_1, G^{i_1})$
and $(\lambda, G^{i_2}) \sim (\lambda_2, G^{i_2})$. From the above construction, 
we get 
a graded $(\Lambda / \sqrt{\Lambda e \Lambda})^e$-module
homomorphism 
$$
G /G^{i_1} \sharp k[V^{i_1}] \oplus G /G^{i_2} \sharp k[V^{i_2}] 
\xrightarrow{\bar\theta^{i_1}_{\lambda} \oplus 
(-\bar\theta^{i_2}_{\lambda})} G /K^{i_1, i_2} \sharp k[V^{i_1, i_2}],
$$ 
which we denote by $\bar\theta^{i_1, i_2}_{\lambda}$. 
Consider the following homomorphism of direct sum of 
graded $(\Lambda / \sqrt{\Lambda e \Lambda})^e$-modules 
$$
\bar{\varpi}_1 := \bigoplus\limits_{(\lambda, G^{i_1}), (\lambda, G^{i_2}) \in \tilde{G}_0} 
\bar\theta^{i_1, i_2}_{\lambda}:  \bigoplus_{(\chi, G^i) \in \tilde{G}_0} \Lambda_i 
\rightarrow \bigoplus\limits_{(\lambda, G^{i_1}), (\lambda, G^{i_2}) \in \tilde{G}_0} G 
/ K^{i_1, i_2} \sharp k[V^{i_1, i_2}]. 
$$
Here, note that we only choose one from $\{\bar\theta^{i_1, i_2}_{\lambda}, 
\bar\theta^{i_2, i_1}_{\lambda}\}$ for the construction of $\bar{\varpi}_1$. 
We have the following.

\begin{lemma}\label{Firstgj}
$\mathrm{Im}(\bar{\varpi}_0) = \mathrm{Ker}(\bar{\varpi}_1)$.
\end{lemma}

\begin{proof}
First, we show that $\varpi_1 \circ \varpi_0 = 0$, which implies that $ \mathrm{Im}
(\varpi_0) \subseteq \mathrm{Ker}(\varpi_1)$. 
By the definitions of $\varpi_0$, $\varpi_1$ and $\bar\theta^{i_1, i_2}_{\lambda}$, to prove 
$\varpi_1 \circ \varpi_0 = 0$, it suffices to prove that 
\begin{align}\label{cinbf4563} 
\bar\theta^{i_1, i_2}_{\lambda} \circ 
( \eta^{i_1}_{\lambda} \oplus \eta^{i_2}_{\lambda} ) = 
(  \bar\theta^{i_1}_{\lambda} \oplus (-\bar\theta^{i_2}_{\lambda}) ) \circ 
( \eta^{i_1}_{\lambda} \oplus \eta^{i_2}_{\lambda} ) = 0
\end{align}
for any $(\lambda, G^{i_1}), 
(\lambda, G^{i_2}) \in \tilde{G}_0$.
This is equivalent to showing that
$\bar\theta^{i_1}_{\lambda} 
\circ \eta^{i_1}_{\lambda} = 
\bar\theta^{i_2}_{\lambda} \circ 
\eta^{i_2}_{\lambda}$. 
In fact, by
the definitions of $\bar\theta^{i_1}_{\lambda}$ and
$\eta^{i_1}_{\lambda}$, it is straightforward
to see that 
$$
\bar\theta^{i_1}_{\lambda} \circ \eta^{i_1}_{\lambda} = 
\eta^{i_1, i_2}_{\lambda}. 
$$ 
Similarly, we have $\bar\theta^{i_2}_{\lambda} \circ \eta^{i_2}_{\lambda} 
= \eta^{i_1, i_2}_{\lambda}$. 
Then it follows that
$\bar\theta^{i_1}_{\lambda} \circ \eta^{i_1}_{\lambda} 
= \bar\theta^{i_2}_{\lambda} \circ \eta^{i_2}_{\lambda}$, and
hence, (\ref{cinbf4563}) holds. 

Now, to prove this lemma, it is left to show that
$\mathrm{Ker}(\varpi_1) \subseteq \mathrm{Im}
(\varpi_0)$. 
To this end, let $\{e_r\}_r$ be the set of indecomposable
idempotents of $\Lambda$. Then $\sum_r e_r$ is the unit of $\Lambda$,
and we have
$$\Lambda / \sqrt{\Lambda e \Lambda} 
\cong \sum_r e_r \Lambda 
\Big/ (\sum_r e_r \sqrt{\Lambda e \Lambda})  
\cong \bigoplus\limits_r (e_r 
\Lambda \big/ e_r \sqrt{\Lambda e \Lambda})$$ 
as graded $\Lambda/\sqrt{\Lambda e \Lambda}$-modules. 
Viewing
$\bar{\varpi}_0$ and 
$\bar{\varpi}_1$ as graded 
$(\Lambda/\sqrt{\Lambda e \Lambda}) 
\otimes \Lambda^{op}$-module morphisms, 
by applying the functor 
$e_r \Lambda \otimes_{\Lambda}(-)$, 
we get the following graded 
$\Lambda /\sqrt{\Lambda e \Lambda}$-module
homomorphisms: 
$$
\varpi_0^r : =  e_r \Lambda \otimes_{\Lambda} \varpi_0: 
R / I_r \rightarrow \bigoplus_{(\chi, G^i) \in \tilde{G}_0} e_r 
\Lambda \otimes_{\Lambda} (G/G^i \sharp k[V^i]) \cong 
\bigoplus_{\substack{(\chi_r, G^i) \in \tilde{G}_0 
}} k[V^i],
$$
and 
\begin{align*}
	\varpi_1^r : =  e_r \Lambda \otimes_{\Lambda} \varpi_1: &  
\bigoplus_{(\chi, G^i) \in \tilde{G}_0} e_r \Lambda \otimes_{\Lambda} 
(G/G^i \sharp k[V^i]) \cong \bigoplus_{\substack{(\chi_r, G^i) 
\in \tilde{G}_0 
}} k[V^i]  \\
& \longrightarrow \bigoplus\limits_{(\lambda, G^{i_1}), (\lambda, G^{i_2}) 
\in \tilde{G}_0} e_r \Lambda \otimes_{\Lambda} (G /K^{i_1, i_2} \sharp k[V^{i_1, i_2}]) \\
&\quad\quad \cong \bigoplus\limits_{\substack{(\chi_r, G^{i_1}), (\chi_r, G^{i_2}) \in \tilde{G}_0  
}} k[V^{i_1, i_2}], 
\end{align*}
where $\chi_r$ is the character corresponding to
idempotent $e_r$ and 
$I_r := e_r \Lambda \otimes_{\Lambda} 
\sqrt{\Lambda e \Lambda} \subseteq e_r \Lambda \cong R$ 
is the 
ideal of $R$. In the above,
recall that $k[V^{i}] \cong R/I^i$ and
$k[V^{i_1, i_2}] 
\cong k[V^{i_1} \cap V^{1_2}] 
\cong R/(I^{i_1}, I^{i_2})$. 

To prove 
$\mathrm{Ker}(\varpi_1)\subseteq
\mathrm{Im}(\varpi_0)$,
it is enough to show that 
$\mathrm{Ker}(\varpi^r_1) \subseteq \mathrm{Im}(\varpi^r_0)$ 
for any $e_r$.
By Lemma \ref{NNned1}, it is obvious that
$\varpi_0^r$ gives a covering
$$\bigsqcup\limits_{\substack{(\chi_r, G^i) \in \tilde{G}_0 
}} \mathrm{Spec}(R/I^i)$$ of $\mathrm{Spec}(R/I_r)$. 
Here, note that for any pair
$$\{ \mathrm{Spec}(R/I^{i_1}), 
\mathrm{Spec}(R/I^{i_2})\} \subseteq \Big\{ \mathrm{Spec}(R/I^i)  
\Big\}_{\substack{(\chi_r, G^i) \in \tilde{G}_0 
}},$$
the intersection of them is exactly 
$\mathrm{Spec}(R/(I^{i_1}, I^{i_2}))$ in $\mathrm{Spec}(R/I_r)$. 
Thus, we obtain the following exact sequence
\begin{align}\label{gnkbvj}
0 \longrightarrow  R/I_r \xrightarrow{\varpi^r_0}  \bigoplus_{\substack{(\chi_r, G^i) 
\in \tilde{G}_0 
}} R/ I^{i} \xrightarrow{\varpi^r_1} 
\bigoplus\limits_{\substack{(\chi_r, G^{i_1}), (\chi_r, G^{i_2}) \in \tilde{G}_0  
}} R/(I^{i_1}, I^{i_2})
\end{align}
of graded 
$\Lambda / \sqrt{\Lambda e \Lambda}$-modules. 
It follows that $\mathrm{Ker}(\varpi^r_1) 
\subseteq \mathrm{Im}(\varpi^r_0)$,
and the lemma is proved.
\end{proof}

From now on, we use $W_G^0$ and $W_G^1$ to denote 
the graded $(\Lambda /\sqrt{\Lambda e \Lambda})^e$-modules 
\begin{equation}
\label{defofWG0}
\bigoplus\limits_{(\chi, G^i) \in \tilde{G}_0} \Lambda_i\quad\mbox{and}\quad
\bigoplus\limits_{(\lambda, G^{i_1}), (\lambda, G^{i_2}) \in \tilde{G}_0}  
K^{i_1, i_2} \sharp k[V^{i_1, i_2}]
\end{equation}
respectively. 
By taking the direct sum of
the
exact sequences (\ref{gnkbvj}) for any $e_r$, we have the 
exact sequence
of graded $(\Lambda /\sqrt{\Lambda e \Lambda})^e$-modules: 
\begin{align}\label{ibyvkl}
0 \longrightarrow \Lambda/ \sqrt{ \Lambda e 
\Lambda }\xrightarrow{\bar{\varpi}_{0}} 
W_G^0 \xrightarrow{\bar{\varpi}_{1}} W_G^1.
\end{align}
Furthermore, for any $m \in \mathbb{N^+}$, 
let $W_G^m$
be the
graded $(\Lambda /\sqrt{\Lambda e \Lambda})^e$-module:
\begin{equation}\label{WGM}
 \bigoplus\limits_{\substack{\{(\lambda, G^{i_l})\}_{1 \leq l \leq m} 
 \subseteq \tilde{G}_0}} (G / K^{i_1, \cdots, i_m} )\sharp k[V^{i_1} \cap \cdots \cap V^{i_m}],
\end{equation}
where $K^{i_1, \cdots, i_m} \subseteq G$ 
is the maximal
subgroup whose invariant space is
$V^{i_1} \cap \cdots \cap V^{i_m}$.
Then there is 
a graded $(\Lambda /\sqrt{\Lambda e \Lambda})^e$-module homomorphism 
$$
\varpi_{m}: W_G^m  \rightarrow W_G^{m+1},
$$
which is defined as follows. 

By the definition
\eqref{WGM}
of $W_G^m$, to define 
$\varpi_{m}$, it is enough to define
its restrictions on the summands 
\begin{equation}
\label{map:pim}
(G / K^{i_1, \cdots, i_m} 
)\sharp k[V^{i_1} \cap 
\cdots \cap V^{i_m}] \twoheadrightarrow  
(G / K^{i_1, \cdots, i_{m+1}} )\sharp 
k[V^{i_1} \cap \cdots \cap V^{i_m} \cap 
V^{i_{m+1}}],
\end{equation}
for all pairs 
$\{(\lambda, G^{i_l})\}_{1 \leq l \leq m+1} \subseteq \tilde{G}_0$. 
For this purpose, we first know that
$V^{i_1} \cap \cdots \cap V^{i_m} 
\cap V^{i_{m+1}} \subseteq V^{i_1} \cap \cdots  \cap V^{i_m}$ as affine subspaces of $V$, and
$K^{i_1, \cdots, i_{m}} \subseteq K^{i_1, \cdots, i_{m+1}}$
as subgroups of $G$. 
Then there is an injection $V^{i_1} \cap \cdots \cap V^{i_m} 
\cap V^{i_{m+1}} \hookrightarrow V^{i_1} \cap \cdots  \cap V^{i_m}$ of affine spaces and surjection of groups
$G / K^{i_1, \cdots, i_m} \twoheadrightarrow G / K^{i_1, \cdots, i_{m+1}}$. 
Analogous to the definition of $\bar\varsigma^j_{\chi}$ in Lemma \ref{Compos}, 
the map \eqref{map:pim}
is then given by the tensor product
of
the natural algebra surjection 
$$k( G / K^{i_1, \cdots, i_m}) \twoheadrightarrow k(G / K^{i_1, \cdots, i_{m+1}})$$ 
of group algebras
and the canonical graded algebra projection 
$$k[V^{i_1} \cap \cdots \cap V^{i_m}] \twoheadrightarrow k[V^{i_1} \cap \cdots \cap V^{i_m} 
\cap V^{i_{m+1}}].$$ 
Here, we replace $G$ and $H$ in 
Lemma \ref{Compos} with 
$G/ K^{i_1, \cdots, i_m}$ 
and $K^{i_1, 
\cdots, i_{m+1}} / K^{i_1, \cdots, i_m}$ respectively.

Using the same approach as 
in the proof of Lemma \ref{Firstgj}, we have the following.  

\begin{lemma}\label{qjsgkn}
Let $W_G^i$, $i=0, 1,\cdots, t$
be the
graded $(\Lambda/\sqrt{ \Lambda e \Lambda})^e$-modules \eqref{WGM}
as in the proof
of Lemma \ref{Firstgj}. 
Then the following  $(\Lambda/ \sqrt{\Lambda e \Lambda})^e$-module sequence 
$$
0 \longrightarrow \Lambda/ \sqrt{\Lambda e \Lambda}\xrightarrow{\varpi_{0}} W_G^0 
\xrightarrow{\varpi_{1}} W_G^1 \xrightarrow{\varpi_{2}}   \cdots   \xrightarrow{\varpi_{t}}     
W_G^t \longrightarrow 0
$$
is exact. 
\end{lemma}

\begin{proof}
By Lemma \ref{NNned1}
$\varpi_0$ is injective;
for the rest $\varpi_i$, where
$i=1,2,\cdots, t$,
$$
\mathrm{Im}(\bar{\varpi}_i) 
= \mathrm{Ker}(\bar{\varpi}_{i+1})
$$
holds by 
the same method as in the proof of Lemma \ref{Firstgj}. 
\end{proof}

Note that if we view 
$W_G^{0}$ is a graded algebra, 
then
$W_G^{m}$ is a graded 
$W_G^{0}$-module for 
all $m\in\{1, 2,\cdots, t\}$.

Now, by Lemma \ref{qjsgkn}, we know that 
$\Lambda/\sqrt{\Lambda e \Lambda}$ is quasi-isomorphic to the above
complex $W_G^{\bullet}$ in $D^{b}(\mathrm{grmod}(\Lambda / \sqrt{\Lambda e \Lambda})^e )$. 
Then, for any object $X$ of 
$D^{b}(\mathrm{grmod}(\Lambda / \sqrt{\Lambda e \Lambda}) )$, 
we have 
\begin{equation}
\label{equivofX}
X \cong X \otimes_{\Lambda / \Lambda e \Lambda}^{\mathbb{L}} \Lambda / 
\Lambda e \Lambda \cong X \otimes_{\Lambda / \Lambda e \Lambda}^{\mathbb{L}} 
W_G^{\bullet} \cong \tilde{X} \otimes_{\Lambda / \Lambda e \Lambda} W_G^{\bullet} 
\end{equation}
in $D^{b}(\mathrm{grmod}(\Lambda / \sqrt{ \Lambda e \Lambda}) )$, 
where $\tilde{X}$ is a finitely generated graded  projective 
$\Lambda / \sqrt{\Lambda e \Lambda}$-module resolution of $X$.

\subsection{Classical
generator of
$D^b(\mathrm{tails}(
\Lambda/\sqrt{\Lambda e\Lambda}
))$}
Retain the notations from the previous subsection.
We show 
the following key lemma.  

\begin{lemma}\label{hidjks}
For any $X$ in $D^{b}(\mathrm{grmod}(\Lambda / \sqrt{\Lambda e \Lambda}) )$,
$$
X \in \big< \bigoplus\limits_{j= l_X}^{n_X} W_G^{0}(j) \big>_{r},
$$
for some $r \in \mathbb{N}$ and some
$l_X, n_X \in \mathbb{Z}$. 	
\end{lemma}

\begin{proof}By
the above isomorphism
\eqref{equivofX},
to prove this lemma, it suffices to show that 
$$
\tilde{X} \otimes_{\Lambda / \sqrt{\Lambda e \Lambda}} W_G^{\bullet} 
\in \big< \bigoplus\limits_{j= l_X}^{n_X} W_G^{\bullet}(j) \big>_{r}. 
$$
To this end, without loss of generality, we 
assume that $X$ is a graded 
$\Lambda /\sqrt{\Lambda e \Lambda}$-module and $\tilde{X}$ is a graded, finitely 
generated free $\Lambda / \sqrt{\Lambda e \Lambda}$-module resolution of $X$. 

First notice
that $\tilde{X} \otimes_{\Lambda /\sqrt{\Lambda e \Lambda}} 
W_G^{\bullet}$ is a bounded above cochain complex. Write it in the form: 
\begin{align}\label{ajsvfd}
\cdots \xrightarrow{\delta^{-1}} M^0 \xrightarrow{\delta^0} M^1 
\xrightarrow{\delta^1} \cdots \xrightarrow{\delta^{s-1}} M^s \longrightarrow 0, 
\end{align}
where for each $i$,
 $M^i$ is a finitely generated graded $W_G^{0}$-module since 
 $\tilde{X}$ is a complex of free $\Lambda / \sqrt{\Lambda e \Lambda}$-modules. 
Note that $W_G^{0}$ is a graded algebra and for any $i$, 
the graded $\Lambda /\sqrt{\Lambda e \Lambda}$-module structure of 
$M^i$ is given by the algebra homomorphism 
$\varpi_0$ and its graded $W_G^{0}$-module structure. 

Since $ X \cong \tilde{X} \otimes_{\Lambda / \sqrt{\Lambda e \Lambda}} 
W_G^{\bullet}$ in $D^{b}(\mathrm{grmod}(\Lambda / \sqrt{\Lambda e \Lambda}) )$, 
we have $H^0( \tilde{X} \otimes_{\Lambda / \sqrt{\Lambda e \Lambda}} 
W_G^{\bullet} ) \cong X$ and $ H^i( \tilde{X} \otimes_{\Lambda / 
\sqrt{\Lambda e \Lambda}} W_G^{\bullet} ) \cong 0$ for $i \neq 0$ 
 as graded 
 $\Lambda /\sqrt{\Lambda e \Lambda}$-modules.  
It follows that the following complex
\begin{align}\label{kdnkd}
0 \longrightarrow M^0 / \mathrm{Im}(\delta^{-1})  \xrightarrow{\tilde{\delta}^0} M^1 
\xrightarrow{\delta^1} \cdots \xrightarrow{\delta^{s-1}} M^s \longrightarrow 0
\end{align}
is isomorphic to \eqref{ajsvfd} in $D^{b}(\mathrm{grmod}(\Lambda / \sqrt{\Lambda e \Lambda}) )$, 
where $\tilde{\delta}^0$ is induced by $\delta^0$ naturally. 

Next, it is direct to see that
in \eqref{kdnkd},
$W_G^{0}$ 
is a homologically 
smooth algebra. Hence we get that 
\begin{align*}
M^i \in \big< \bigoplus\limits_{j= l_i}^{n_i} W_G^{0}(j) \big>_{r_i}
\end{align*}
in $D^{b}(\mathrm{grmod}(\Lambda / \sqrt{\Lambda e \Lambda}) )$ for some
$r_i \in \mathbb{N}$ and $l_i, n_i \in \mathbb{Z}$.
Thus to prove this lemma, it is enough to prove that 
\begin{align}\label{dbdoifnb}
M^0 / \mathrm{Im}(\delta^{-1}) \in \big< \bigoplus\limits_{j= l'}^{n'} W_G^{0}(j) \big>_{r'}
\end{align}
in $D^{b}(\mathrm{grmod}(\Lambda / \sqrt{\Lambda e \Lambda}) )$ for some
$r' \in \mathbb{N}$ and $l', n' \in \mathbb{Z}$. 

In fact, from the construction of complex $\tilde{X} \otimes_{\Lambda / 
\sqrt{\Lambda e \Lambda}} W_G^{\bullet}$ (see (\ref{ajsvfd})), we know that 
\begin{align*}
M^0 \in \mathrm{add}(\bigoplus_{i = 0}^{s} 
\bigoplus\limits_{j= \bar{l}_i}^{\bar{n}_i} W_G^{i}(j)  )
\end{align*}
as graded $W_G^{\bullet}$-modules, for some $\bar{l}_i, \bar{n}_i \in \mathbb{Z}$. 
Now, suppose that $G/K^l \sharp k[W^l] \subseteq M^0$ as a direct summand. 
Then there is a natural projection of $W_G^0$-modules. 
\begin{align*}
\pi_{K^l}:  M^0 \twoheadrightarrow  G/K^l \sharp k[W^l].
\end{align*}
It induces another graded $W_G^0$-modules projection 
\begin{align}\label{jsverte}
M^0 / \mathrm{Im}(\delta^{-1}) \twoheadrightarrow  (G/K^l \sharp k[W^l] 
) / \pi_{K^l}(\mathrm{Im}(\delta^{-1})). 
\end{align}
Since $W_G^{0}$ is a homologically smooth algebra, 
$(G/K^l \sharp k[W^l] ) / \pi_{K^l}(\mathrm{Im}(\delta^{-1})) 
\in \mathrm{Perf}(W_G^{0})$. It implies that 
\begin{align*}
(G/K^l \sharp k[W^l] ) / \pi_{K^l}
(\mathrm{Im}(\delta^{-1})) \in 
\big< \bigoplus\limits_{j= \bar{l}}^{\bar{n}} W_G^{0}(j) \big>_{\bar{r}}
\end{align*}
in $D^{b}(\mathrm{grmod}(\Lambda / \sqrt{\Lambda e \Lambda}) )$,
for some $\bar{r} \in \mathbb{N}$ and $\bar{l}, \bar{n} \in \mathbb{Z}$. 
Note that the kernel of the morphism (\ref{jsverte}) is exactly 
$\mathrm{Ker}(\pi_{K^l}) / 
( \mathrm{Ker}(\pi_{K^l}) 
\cap \mathrm{Im}(\delta^{-1}) )$. 
Thus to prove \eqref{dbdoifnb}, 
it is enough to prove that 
\begin{align}\label{dbdoifnbddd}
\mathrm{Ker}(\pi_{K^l}) / 
(\mathrm{Ker}(\pi_{K^l}) 
\cap \mathrm{Im}(\delta^{-1})) \in 
\big< \bigoplus\limits_{j= \tilde{l}}^{n'} 
W_G^{0}(j) \big>_{r'}
\end{align}
in $D^{b}(\mathrm{grmod}(\Lambda / \sqrt{\Lambda e \Lambda}) )$ 
for some $r' \in \mathbb{N}$ and $\tilde{l}, n' \in \mathbb{Z}$. 

Summarizing
the above argument, we have
reduced the proof of
\eqref{dbdoifnb}
to the proof of
\eqref{dbdoifnbddd}.
On the other hand, 
by the definition of $\pi_{K^l}$, 
we also have
\begin{align}\label{sjbdfggh11}
\mathrm{Ker}(\pi_{K^l}) \in \mathrm{add}
(\bigoplus_{i = 0}^{s} 
\bigoplus\limits_{j= l_i}^{\bar{n}_i} W_G^{i}(j)).
\end{align} 
Thus combining
\eqref{dbdoifnbddd} and \eqref{sjbdfggh11}
and by the same approach as above, we may
continue to
reduce the proof (\ref{dbdoifnbddd}) 
to a similar but rank smaller
graded $\Lambda / \sqrt{\Lambda e \Lambda}$-module. 
Repeat this process,
and since $\mathrm{Ker}(\pi_{K^l})$ is a finitely 
generated graded
$\Lambda / \sqrt{\Lambda e \Lambda}$-module, the proof
is ultimately reduced to
$G/K^{l'} \sharp 
k[W^{l'}]$ for some pair $(K^{l'}, W^{l'})$. 
But as in the above argument, we
already know that 
\begin{align*}
(G/K^{l'} \sharp k[W^{l'}] ) / (G/K^{l'} \sharp k[W^{l'}]  
\cap \mathrm{Im}(\delta^{-1})) \in 
\big< \bigoplus\limits_{j= \tilde{l}}^{n'} W_G^{0}(j) \big>_{r'}
\end{align*}
in $D^{b}(\mathrm{grmod}(\Lambda / \Lambda e \Lambda) )$ for some 
$r' \in \mathbb{N}$ and $\tilde{l}, n' \in \mathbb{Z}$. 
Thus the lemma is proved.
\end{proof}

By Proposition \ref{serre}, we know that 
$$
D^{b}(\mathrm{grmod}(\Lambda / \sqrt{\Lambda e \Lambda}) 
) / D^{b}_{(\mathrm{tors}(\Lambda / \sqrt{\Lambda e \Lambda}) )}
(\mathrm{grmod}(\Lambda / \sqrt{\Lambda e \Lambda}) )    
\cong D^{b}(\mathrm{tails}(\Lambda / \sqrt{\Lambda e \Lambda}) ). 
$$ 
Then by Lemma \ref{hidjks},
for any $X$ in $D^{b}(\mathrm{tails}(\Lambda / \sqrt{\Lambda e \Lambda}) )$,
$$
X \in \big< \bigoplus\limits_{j= l_X}^{n_X} W_G^{0}(j) \big>_{r},
$$
for some $r \in \mathbb{N}$ and $l_X, n_X \in \mathbb{Z}$ (note that here
$l_X, n_X$ may depend on $X$).

\begin{lemma}\label{mjdbck}
$D^{b}( \mathrm{tails}(\Lambda / \sqrt{\Lambda e \Lambda}) )$ is classically generated
by $$\bigoplus_{(\chi, G^i) \in \tilde{G}_0} 
\bigoplus_{j = 0}^{\mathrm{dim}(V^i) -1} G / G^i \sharp k[V^i](j).$$
\end{lemma}

\begin{proof}
By the above argument, to prove this lemma, it suffices to show that for any 
$l \in \mathbb{Z}$, $$W_G^{0}(l) \in \big< \bigoplus_{(\chi, G^i) \in \tilde{G}_0} 
\bigoplus_{j = 0}^{\mathrm{dim}(V^i) -1} G / G^i \sharp k[V^i](j) \big>_{r_l}$$ for some $r_l \in \mathbb{N}$.  

Recall from \eqref{defofWG0}
that the graded algebra
$W_{G}^{0} =
\bigoplus\limits_{(\chi, G^i) 
\in \tilde{G}_0} \Lambda_i$,
where
$\Lambda_i = G/G^i \sharp k[V^i]$. 
Note that $G/G^i \sharp k[V^i]$ is an AS regular algebra by Proposition \ref{ASstru}.
Next, by Corollary \ref{ydbvwjv}, we have 
for any $l \in \mathbb{Z}$, $$G/G^i \sharp k[V^i] (l) 
\in \big< \bigoplus\limits_{j= 0}^{\mathrm{dim}(V^i)} G/G^i \sharp k[V^i] (j) \big>_{s_l}$$ 
for some $s_l \in \mathbb{N}$.  
It follows that $$W_G^{0}(l) \in \big< \bigoplus_{(\chi, G^i) 
\in \tilde{G}_0} \bigoplus_{j = 0}^{\mathrm{dim}(V^i) -1} G / G^i 
\sharp k[V^i](j) \big>_{r_l}$$ for some $r_l \in \mathbb{N}$.  
\end{proof}

As a corollary, we get
the following.

\begin{corollary}
\label{cor:generator}
The following object
 $$\bigoplus_{(\chi, G^i) \in \tilde{G}_0} 
 \bigoplus_{j = 0}^{\mathrm{dim}(V^i) -1} G / G^i \sharp k[V^i](j)$$
classically generates
$D^{b}_{\mathrm{tails}
(\Lambda / \sqrt{\Lambda e \Lambda})}
( \mathrm{tails}(\Lambda))$.
\end{corollary}

\begin{proof}
Combine 
Lemmas \ref{compgene} and \ref{mjdbck}.
\end{proof}

\subsection{From
$D^b_{
\mathrm{tails}(\Lambda/
\sqrt{\Lambda e\Lambda})
}(\mathrm{tails}(\Lambda))$
to $D^b_{\mathrm{tails}
(\Lambda/\Lambda e\Lambda)}
(\mathrm{tails}(\Lambda))$}

Next, we prove the following lemma.

\begin{lemma}\label{nribe}	
For any object $Y$ of $D^{b}_{\mathrm{tails}(\Lambda/\Lambda e \Lambda)}(\mathrm{tails}(\Lambda))$, 
there is an object, say $X$, of $D^{b}_{\mathrm{tails}(\Lambda/ 
\sqrt{ \Lambda e \Lambda})}(\mathrm{tails}(\Lambda))$ such that 
$$
Y \in \langle X  \rangle_{m}
$$
in $D^{b}(\mathrm{tails}(\Lambda))$, for some $m \in \mathbb{Z}$. 
\end{lemma}

\begin{proof}
Without loss of generality, we 
assume that 
$Y$ is a graded $\Lambda/\Lambda e \Lambda$-module. 
Then to prove this lemma, it
suffices to show that 
$$
Y \in \langle X  \rangle_{m}
$$	
in $D^{b}(\mathrm{tails}(\Lambda / \Lambda e \Lambda))$, for some object $X$ of 
$D^{b}_{\mathrm{tails}(\Lambda/ \sqrt{ \Lambda e \Lambda})}(\mathrm{tails}(\Lambda))$ 
and some $m \in \mathbb{Z}$.   	

First, consider the following
exact sequence of
graded $\Lambda/ \Lambda e \Lambda$-modules 
$$
0 \longrightarrow   \sqrt{ \Lambda e \Lambda} / \Lambda e \Lambda 
\longrightarrow \Lambda/ \Lambda e \Lambda \longrightarrow \Lambda/ \sqrt{ \Lambda e \Lambda} \longrightarrow 0. 
$$
By applying 
$(-)\otimes
_{\Lambda/ \Lambda e \Lambda}
Y$ to it we get another 
exact sequence of graded 
$\Lambda/ \Lambda e \Lambda$-modules
$$
0 \longrightarrow Y_1 \longrightarrow Y \longrightarrow Y 
\otimes_{\Lambda/ \Lambda e \Lambda} \Lambda/ \sqrt{ \Lambda e \Lambda}  \longrightarrow 0,
$$
where
$Y_1$ is a graded submodule of $Y \otimes_{\Lambda/ \Lambda e \Lambda}
\sqrt{ \Lambda e \Lambda} / \Lambda e \Lambda$
and 
$ Y \otimes_{\Lambda/ \Lambda e \Lambda} \Lambda/ \sqrt{ \Lambda e \Lambda}$ is a graded
$\Lambda/ \sqrt{ \Lambda e \Lambda}$-module.  
Hence, to prove this lemma, it 
suffices to show that 
$$
Y_1
\in \langle X  \rangle_{m}
$$	
in $D^{b}(\mathrm{tails}(\Lambda / \Lambda e \Lambda))$ for some object $X$ of 
$D^{b}_{\mathrm{tails}(\Lambda/ \sqrt{ \Lambda e \Lambda})}(\mathrm{tails}(\Lambda))$ 
and some $m \in \mathbb{Z}$.   	

Since $\Lambda e \Lambda$ a is finitely generated graded $\Lambda$ module, 
so is $\sqrt{\Lambda e \Lambda}$.  
Then
by the definition of 
$\sqrt{\Lambda e \Lambda}$ 
(see Definition \ref{bdfbv11267}),
it is straightforward to verify that 
$$\sqrt{\Lambda e \Lambda}^l \subseteq \Lambda e \Lambda,$$ 
for some $l \in \mathbb{N}$,
which implies that 
$$
\sqrt{ \Lambda e \Lambda} / \Lambda e \Lambda \in 
\mathrm{Ob}(\mathrm{mod}( \Lambda / \sqrt{ \Lambda e \Lambda}^{l-1} ) ). 
$$
It also suggests that 
$$
Y_1 \in 
\mathrm{Ob}(\mathrm{mod}( \Lambda / \sqrt{ \Lambda e \Lambda}^{l-1} )). 
$$

Next, we replace $\Lambda / \Lambda e \Lambda$
by $\Lambda / \sqrt{ \Lambda e \Lambda}^{l-1}$, and consider the following
exact sequence of
graded $\Lambda / \sqrt{ \Lambda e \Lambda}^{l-1}$-modules
$$
0 \longrightarrow \sqrt{ \Lambda e \Lambda} / \sqrt{ \Lambda e \Lambda}^{l-1}
\longrightarrow \Lambda  / \sqrt{ \Lambda e \Lambda}^{l-1} 
\longrightarrow \Lambda / \sqrt{ \Lambda e \Lambda} \longrightarrow 0. 
$$
By applying
$(-)\otimes_{\Lambda  
/ \sqrt{ \Lambda e \Lambda}^{l-1}}Y_1$,
we get another exact sequence
of graded $\Lambda
/ \sqrt{ \Lambda e \Lambda}^{l-1}$-modules 
$$
0 \longrightarrow Y_2 \longrightarrow Y_1 \longrightarrow Y_1 \otimes_{\Lambda
/ \sqrt{ \Lambda e \Lambda}^{l-1}} \Lambda/ \sqrt{ \Lambda e \Lambda}  \longrightarrow 0.
$$
By the same argument as for $Y_1$ above, 
we get that 
$$
Y_2 \in 
\mathrm{Ob}(\mathrm{grmod}( \Lambda
/ \sqrt{ \Lambda e \Lambda}^{l-2} )). 
$$
Repeating this process,
we obtain objects $\{Y_i\}_{1 \leq i \leq l-1}$ such that 
$$
Y_i \in \mathrm{Ob}(
\mathrm{grmod}( \Lambda / \sqrt{ \Lambda e \Lambda}^{l-i}))  
$$
and exact sequences
$$
0 \longrightarrow Y_{i} \longrightarrow Y_{i-1} \longrightarrow Y_{i-1} 
\otimes_{\Lambda / \sqrt{ \Lambda e \Lambda}^{l-i}} \Lambda/ \sqrt{ \Lambda e \Lambda}  \longrightarrow 0.
$$
for all $1 \leq i \leq l-1$. 

Since $Y_{l-1} \in 
\mathrm{Ob}( \mathrm{grmod}( \Lambda / \sqrt{ \Lambda e \Lambda} ))$ and 
$$Y_{i-1} \otimes_{\Lambda / \sqrt{ \Lambda e \Lambda}^{l-i}} \Lambda/ \sqrt{ \Lambda e \Lambda} \in 
\mathrm{Ob}(\mathrm{grmod}( \Lambda / \sqrt{ \Lambda e \Lambda} ))$$ for all $1 \leq i \leq l-1$, 
we get that 
$$
Y_1
\in \langle X  \rangle_{m}
$$	
in $D^{b}(\mathrm{tails}(\Lambda / \Lambda e \Lambda))$, for some object $X$ of 
$D^{b}_{\mathrm{tails}(\Lambda/ \sqrt{ \Lambda e \Lambda})}(\mathrm{tails}(\Lambda))$ 
and some $m \in \mathbb{Z}$.   		
\end{proof}

\begin{proof}
[Finishing the proof of 
Theorem \ref{thm:smoothcatcpt2equiv} (=Theorem \ref{thm:smoothcatcpt2}) 
]
According to Corollary \ref{cor:generator},
$$\bigoplus_{(\chi, G^i) \in \tilde{G}_0} 
\bigoplus_{j = 0}^{\mathrm{dim}(V^i) -1} G / G^i \sharp k[V^i](j)$$
classically generates $D^{b}_{\mathrm{tails}
(\Lambda / \sqrt{\Lambda e \Lambda})}( \mathrm{tails}(\Lambda))$,
which, by Lemma \ref{nribe},
then classically 
generates
$D^{b}_{\mathrm{tails}
(\Lambda /\Lambda e \Lambda)}( \mathrm{tails}
(\Lambda))$.
The theorem follows. 
\end{proof}

\section{Smooth categorical compactification of $\mathcal{D}_{sg}^{\mathbb{Z}}(S)$}\label{Fina}

In this section, we prove Theorem \ref{thm:catcptofsing}.

Let us first recall that by Theorem \ref{orlov}, there is a localization functor
$$
\Xi: D^{b}(\mathrm{tails}(S)) \rightarrow D^{\mathbb{Z}}_{sg}(S)
$$
with
$$
\mathrm{Ker}(\Xi) = \big< S, \cdots, S(n-1) \big>.
$$
Hence, there is also a localization functor
$$
\Theta =\Xi \circ ((-) \otimes^{\mathbb{L}}_{\Lambda} \Lambda e ): D^{b}(\mathrm{tails}
(\Lambda)) \rightarrow D^{\mathbb{Z}}_{sg}(S).
$$

\begin{theorem}[Theorem \ref{thm:catcptofsing}]\label{thm:catcptofsingequiv}
Let
$$
\tilde{\Theta}: \mathcal{D}^{b}(\mathrm{tails}(\Lambda)) 
\rightarrow \mathcal{D}_{sg}^{\mathbb{Z}}(S)
$$
be the DG lifting of $\Theta$. Then $\tilde\Theta$
is a smooth categorical compactification of $\mathcal{D}_{sg}^{\mathbb{Z}}(S)$. \end{theorem} 
 
\begin{proof}

It is clear that the functor $\tilde{\Theta}$ is a DG quotient since 
both $\tilde\Xi$ and $(-) \otimes^{\mathbb{L}}_{\Lambda} \Lambda e$ 
are (see Definition \ref{sfkd} and Example \ref{jsvdic}(3)). 
Note that we here have identified 
$D^{\mathbb{Z}}_{sg}(A)$ with $ \Phi(D^{\mathbb{Z}}_{sg}(A))$ 
as triangulated categories (see Theorem \ref{orlov}).
 
In the meantime, by Proposition \ref{gfubjnv}, 
$\mathcal{D}^{b}(\mathrm{tails}(\Lambda)) \cong 
\mathcal{P}\mathrm{erf}(\nabla \Lambda)$ 
is smooth and proper. 
Hence, to prove the theorem
it suffices to show that $\mathrm{Ker}(\Theta)$ admits a classical generator.

Consider the following functor
$$
(-) \otimes^{\mathbb{L}}_{\Lambda} \Lambda 
e: \mathrm{Ker}(\Theta) \rightarrow \mathrm{Ker}(\Xi) 
= \langle S, \cdots, S(n-1)\rangle.
$$
Since
$$
D_{\mathrm{tails}(\Lambda / \Lambda e \Lambda )}^{b}
(\mathrm{tails}(\Lambda)) \subseteq  \mathrm{Ker}(\Theta), 
$$
the kernel of the above functor $(-) \otimes^{\mathbb{L}}_{\Lambda} \Lambda e$ 
in $\mathrm{Ker}(\Theta)$ is also
$$
 D_{\mathrm{tails}(\Lambda / \Lambda e \Lambda )}^{b}(\mathrm{tails}(\Lambda)),
$$
and the above functor is a localization functor.

Meanwhile, 
$\mathrm{add}\{ e \Lambda (i)\}_{i \in \mathbb{Z}}$
are contained in $\mathrm{Ob}(\mathrm{Ker}(\Theta))$.
Now, let $Z$ be an arbitrary object of $\mathrm{Ker}(\Theta)$. Then we have
$$Z \otimes^{\mathbb{L}}_{\Lambda} \Lambda e 
= Z \otimes_{\Lambda} \Lambda e \in \langle S, \cdots, S(n-1)\rangle.$$
Thus by the isomorphism
$$
\mathrm{Hom}_{D^{b}(\mathrm{tails}(S))}(S(i), S(j)[q] ) 
\cong \mathrm{Hom}_{D^{b}(\mathrm{grmod}(S))}(S(i), S(j)[q])
$$
for any $ 0 \leq i, j \leq n-1$ and $q \in \mathbb{Z}$ (see \cite[Theorem 4.14]{MM}),
there are two cochain complexes $\widetilde{Z}^{\bullet}$ and $Y^{\bullet}$ of 
$\mathrm{grmod}(S)$ and two cochain complex morphisms
$$
\varphi_{Z}: \widetilde{Z}^{\bullet} \rightarrow Z \otimes_{\Lambda} \Lambda e
\quad\mbox{and}\quad
\phi_{Z}: \widetilde{Z}^{\bullet} \rightarrow Y^{\bullet},
$$
such that $Y^j \in \mathrm{add}( \bigoplus\limits^{n-1}_{i = 0} S(i) )$ 
for any $j \in \mathbb{Z}$, 
and $\mathrm{Cone}(\phi_Z)$ and $\mathrm{Cone}(\varphi_Z)$ are both trivial in 
$D^{b}(\mathrm{tails}(S))$.
After tensoring with $ (-) \otimes_{S} e \Lambda $ on
the cochain complex level, 
we have the following diagram of cochain complex morphisms:
\begin{align}\label{jcdd}
\xymatrix{
&   \widetilde{Z}^{\bullet} \otimes_{S} e \Lambda      \ar[dl]_{\phi_{Z}  
\otimes_{S} e \Lambda }    \ar[dr]^{\varphi_{Z} \otimes_{S} e \Lambda } \\
Y^{\bullet} \otimes_{S} e \Lambda    & &     Z \otimes_{\Lambda} \Lambda e \otimes_{S} e \Lambda . }
\end{align}
Going back to derived categories, since 
$$( Z \otimes_{\Lambda} \Lambda
e \otimes_{S} e \Lambda ) \otimes^{\mathbb{L}}_{\Lambda} \Lambda e =  Z \otimes_{\Lambda} \Lambda
e \in \mathrm{Ob}(\mathrm{Ker}(\Xi)),
$$
we have that
$Z \otimes_{\Lambda} \Lambda e \otimes_{S} e 
\Lambda\in\mathrm{Ob}(\mathrm{Ker}(\Theta))$. 
Moreover, there is a natural morphism
\begin{align}\label{jcdd11}
Z \otimes_{\Lambda} \Lambda e \otimes_{S} e \Lambda \rightarrow Z,
\end{align}
which is 
induced by the multiplication map 
$\Lambda e \otimes_{S} e \Lambda \rightarrow \Lambda$ in $D^{b}
(\mathrm{grmod}(\Lambda^e))$. 
Denote the above morphism by $i_Z$. We have that
$$
\mathrm{Cone}(i_Z) \otimes^{\mathbb{L}}_{\Lambda} \Lambda e 
\cong \mathrm{Cone}(Z \xrightarrow{\mathrm{Id}} Z)  = 0
$$ 
in $D^{b}(\mathrm{tails}(S))$, and hence 
$$
\mathrm{Cone}(i_Z) \in \mathrm{Ob}(D_{\mathrm{tails}
(\Lambda / \Lambda e \Lambda )}^{b}(\mathrm{tails}(\Lambda))).
$$

In the meantime, we know that
$$
\mathrm{Cone}(\varphi_{Z} \otimes_{S} e\Lambda ) 
\otimes^{\mathbb{L}}_{\Lambda} \Lambda e \cong \mathrm{Cone}(\varphi_{Z}) 
\quad\mbox{and}\quad
\mathrm{Cone}(\phi_{Z} \otimes_{S} e\Lambda )
 \otimes^{\mathbb{L}}_{\Lambda} \Lambda e \cong \mathrm{Cone}(\phi_{Z})
 $$
are both trivial in $D^{b}(\mathrm{tails}(S))$. It follows that 
$\mathrm{Cone}(\varphi_{Z} \otimes_{S} e\Lambda )$ and 
$\mathrm{Cone}(\phi_{Z} \otimes_{S} e\Lambda )$
are both objects of $D_{\mathrm{tails}(\Lambda / \Lambda e \Lambda )}^{b}(\mathrm{tails}(\Lambda))$.

Finally, if we view $Y$ as an object in $D^{b}(\mathrm{grmod}(\Lambda))$,
then
$$
Y^j \in \mathrm{add}(\bigoplus^{n-1}_{i =0} S(i) )
$$
for any $j$. Thus we have
$Y^j \otimes_{S} e \Lambda \in \mathrm{add}(\bigoplus\limits^{n-1}_{i = 0} e \Lambda(i) 
)$ in $D^{b}(\mathrm{grmod}(\Lambda))$. 
It implies that 
$Y^j \otimes_{S} e \Lambda \in \mathrm{add}(\bigoplus\limits^{n-1}_{i = 0} e \Lambda(i) 
)$ in $D^{b}(\mathrm{tails}(\Lambda))$.

Summarizing the above results, we obtain that $\mathrm{Cone}(i_{Z})$,
$\mathrm{Cone}(\varphi_{Z} \otimes_{S} e\Lambda )$ and 
$\mathrm{Cone}(\phi_{Z} \otimes_{S} e\Lambda )$
are all objects of $D_{\mathrm{tails}(\Lambda / \Lambda e \Lambda )}^{b}
(\mathrm{tails}(\Lambda))$, and
$Y \otimes_{S} e \Lambda \in \mathrm{add}(\bigoplus\limits^{n-1}_{i = 0} 
e \Lambda(i) )$ 
in $D^{b}(\mathrm{tails}(\Lambda))$.
In the proof of Theorem \ref{thm:smoothcatcpt1}, we know that
$D_{\mathrm{tails}(\Lambda / \Lambda e \Lambda )}^{b}(\mathrm{tails}(\Lambda))$
admits a classical generator, which we denote by $\mathcal{P}$. 
Thus, by the diagram (\ref{jcdd})  and the morphism $i_Z$ (see
\ref{jcdd11}), we
have $Z \in \big< (\bigoplus\limits^{n-1}_{i = 0} e \Lambda(i)) \oplus \mathcal{P} \big>_{r_Z}$
for certain $r_{Z} \in \mathbb{Z}$,
and therefore, $(\bigoplus\limits^{n-1}_{i = 0} e \Lambda(i)) 
\oplus \mathcal{P}$ is a classical generator
of $\mathrm{Ker}(\Theta)$.
\end{proof}

\section{Homotopical finite presentation of two categories}\label{sect:HFP}

In this section, we give an application of our main theorems.
We first recall the notion of homotopically finitely presented algebras,
which
was introduced by To\"en and Vaqui\'e in \cite{TV},
in their study of the moduli stack of DG categories. 

\begin{definition}[Homotopical finite presentation; see \cite{TV}]
A DG algebra $A$ is called {\it homotopically finitely presented} if
in the homotopy category of DG algebras, $A$ is a retract of a free
graded algebra $k\langle x_1, x_2, \cdots, x_m\rangle$ with differential satisfying that
$$
d(x_i) \in k\langle x_1, \cdots, x_{i-1}\rangle, \quad\mbox{for all} \,\, i.
$$
A small DG category $\mathcal{C}$ is homotopically finitely presented if it is Morita
equivalent to a DG algebra $A$ which is homotopically finitely presented.
\end{definition}

In \cite{Efi1} Efimov showed the following.

\begin{proposition}[{\cite[Proposition 2.8]{Efi1}}]\label{fhp}
$(1)$ If a DG category $\mathcal{A}$ admits a smooth categorical compactification, then
$\mathcal{A}$ is  homotopically finitely presented.

$(2)$ In particular, if
a DG category $\mathcal{A}$ admits a smooth categorical compactification, then
$\mathcal{A}$ is smooth.
\end{proposition}

Thus as corollary to Efimov's result, we have the following.

\begin{corollary}\label{cor:main1and2}
Let $S=R^G$ as before. Then
the DG categories $\mathcal{D}^{b}(\mathrm{tails}(S))$ and 
$\mathcal{D}_{sg}^{\mathbb{Z}}(S)$ are both homotopically finitely presented and hence are smooth.
\end{corollary}

\begin{proof}
By Theorems \ref{thm:smoothcatcpt1} and \ref{thm:catcptofsing}, both 
$D^{b}(\mathrm{tails}(S))$ and 
$D_{sg}^{\mathbb{Z}}(S)$ admit a smooth categorical compactification, and then
by Proposition 
\ref{fhp}, both of them
are homotopically finitely presented. 
\end{proof}

\section{Examples}\label{EA}

In this section, we study three examples, which may be viewed as the applications
of the results in the previous sections.

\begin{example}
Let $V = k^{3}$, and let $G \subseteq \mathrm{SL}(3, k)$ be the finite group generated by elements
$f_1 = \mathrm{diag}(1, -1, -1)$ and $f_{2} = \mathrm{ding}(-1, 1, -1)$, which is an Abelian group.

Consider the graded singularity category of
$S: =R^{G}$, where $R := k[x_1, x_2, x_3]$.
It is straightforward to see that
$$
S\cong k[x^{2}_{1}, x^{2}_{2}, x^{2}_{3}, x_{1}x_{2}x_{3}] \cong  k[x, y, z, w]/(xyz-w^2)
$$
where $x = x^{2}_1$, $y = x^{2}_2$, $z = x^{2}_3$ and $w = x_1 x_2 x_3$. 

We next construct the $S$-algebra $\Lambda$.
From above argument, we know that $G$ has three irreducible 
1-dimensional representations $V^1, V^2$ and $V^3$
with generators $E_1, E_2$ and $E_3$, such that
\begin{align*}
f_{1} \ast E_1 &=  E_1, \quad f_{2} \ast E_1 = -E_1,\\
f_{1} \ast E_2 &=  -E_2, \quad f_{2} \ast E_2 =  E_2,\\
f_{1} \ast  E_3 &= -E_3, \quad f_{2} \ast E_3 = - E_3,
\end{align*}
where the symbol $\ast$ represents the $G$-action.
Thus, $\Lambda$ is given by
$$\mathrm{End}_{S}( (V^1 \otimes R)^{G} \oplus 
(V^2 \otimes R)^{G} \oplus (V^3 \otimes R)^{G} ) \cong
G \sharp R. $$
It can be described by following quiver:
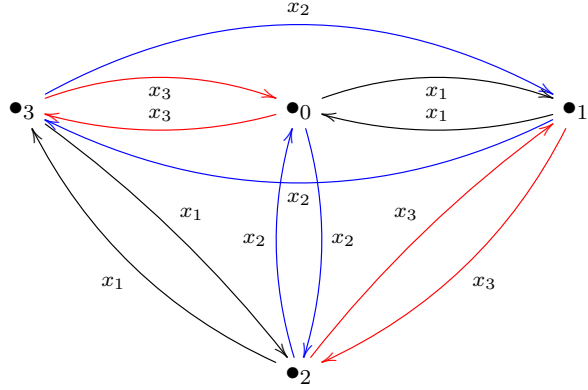
\begin{displaymath}
\xymatrix{
\bullet_{3} \ar@[red]@[red]@/^0.4cm/[rrr]_{x_{3}} \ar@[blue]@/^1.1cm/[rrrrrr]^{x_{2}}
\ar@/^0.20cm/[drdrdr]^{x_{1}}
&&& \bullet_{0} \ar@/^0.4cm/[rrr]_{x_{1}}
\ar@[red]@/^0.3cm/[lll]_{x_{3}}
\ar@[blue]@/^0.3cm/[ddd]^{x_{2}}
&&& \bullet_{1} \ar@/^0.3cm/[lll]_{x_{1}}
\ar@[blue]@/^1cm/[llllll]^{x_{2}}
\ar@[red]@/^0.50cm/[dldldl]^{x_{3}} \\ \\ \\ 
&&& \bullet_{2} \ar@[blue]@/^0.3cm/[uuu]^{x_{2}}
\ar@/^0.50cm/[ululul]^{x_{1}}
\ar@[red]@/^0.20cm/[ururur]^{x_{3}}
}
\end{displaymath}
In the above diagram, the vertex $0$ corresponds to Cohen-Macaulay module $R^G$ and
idempotent $e$; the vertex $1$ corresponds to Cohen-Macaulay module $(V^1 \otimes R)^G$ and
idempotent $e_1$; the vertex $2$ corresponds to Cohen-Macaulay module $(V^2 \otimes R)^G$ and
idempotent $e_2$ ; the vertex $1$ corresponds to Cohen-Macaulay module $(V^3 \otimes R)^G$ and
idempotent $e_3$.

The quiver of 
$\Lambda / \Lambda e \Lambda$,
denoted by $\bar Q$,
is the following: 
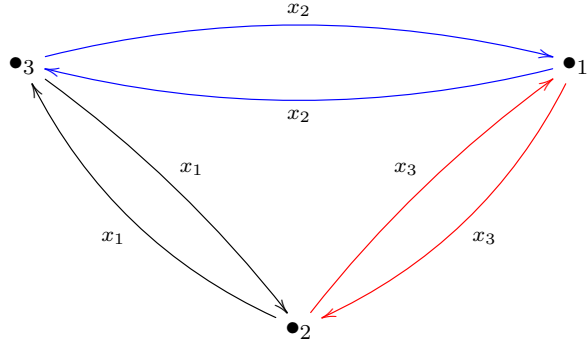
\begin{displaymath}
\xymatrix{
\bullet_{3} \ar@[blue]@/^0.5cm/[rrrrrr]^{x_{2}}
\ar@/^0.20cm/[drdrdr]^{x_{1}}
&&& 
&&& \bullet_{1} 
\ar@[blue]@/^0.5cm/[llllll]^{x_{2}}
\ar@[red]@/^0.50cm/[dldldl]^{x_{3}} \\ \\ \\ 
&&& \bullet_{2} 
\ar@/^0.50cm/[ululul]^{x_{1}}
\ar@[red]@/^0.20cm/[ururur]^{x_{3}}
}
\end{displaymath}
We thus have
the quiver $\tilde{Q}$ of 
$\Lambda / \sqrt{\Lambda e \Lambda}$ is same with $\bar{Q}$ since 
$\Lambda e \Lambda \cong \sqrt{\Lambda e \Lambda}$ as  ideals of $\Lambda$. 
We see that 
$$
\tilde{G}_0 = \big\{ \{\chi_1,  G^2\}, \{\chi_2,  G^3\}, \{\chi_3,  G^1\}  \big\},
$$
where the character $\chi_i$ is corresponding to idempotent $e_i$, $G^1= \langle f_1\rangle$, $G^2 =\langle f_2\rangle$, and $G^3 =\langle f_1 f_2\rangle$.

We next give the (classical) generators
of $D^{b}(\mathrm{tails}(S))$ and $D_{sg}^{\mathbb{Z}}(S)$. 
Notice that
there is a categorical smooth compactification 
of $D^{b}(\mathrm{tails}(S))$ such that
$$
(-) \otimes_{\Lambda}^{\mathbb{L}} 
\Lambda e : D^{b}(\mathrm{tails}(\Lambda)) 
\rightarrow D^{b}(\mathrm{tails}(S))
$$
and a categorical smooth compactification of $D_{sg}^{\mathbb{Z}}(S)$ such that
$$
\Theta: D^{b}(\mathrm{tails}(\Lambda)) \rightarrow D_{sg}^{\mathbb{Z}}(S).
$$
By Theorem \ref{thm:smoothcatcpt2}, the generator $X$ of 
$\mathrm{Ker}( (-) \otimes_{\Lambda}^{\mathbb{L}} \Lambda e )$ is 
\begin{align*}
X  = (\bigoplus\limits_{0\leq i \leq 2} G/G^1 \sharp k[V^{1}](i)) 
 \oplus (\bigoplus\limits_{0\leq i \leq 2}G/G^2 \sharp k[V^{2}](i)) \oplus 
 (\bigoplus\limits_{0\leq i \leq 2}G/G^3 \sharp k[V^{3}](i)).
\end{align*}  
Moreover, 
the generator $Y$ of $\mathrm{Ker}(\Theta)$ is 
$$
Y = X \oplus (\bigoplus_{0 \leq i \leq 2}  e \Lambda (i)).
$$
\end{example}

\begin{example}\label{example2}
Let $V = k^{3}$,
where $k$ is also assumed to be algebraically
closed. Let
$G \subseteq \mathrm{SL}(3, k)$ be the finite group generated by element 
$g = \mathrm{diag}(\sigma, \sigma, \sigma^{2})$, where $\sigma$ is a 4-th prime root of the unit. 

Consider the graded singularity category of
$S: =R^{G}$, where $R := k[x_1, x_2, x_3]$ with canonical grading such that 
$\mathrm{deg}(x_i) = 1$ for any $i$.
Observe that
\begin{align*}
S & = k[x^{4}_{1}, x^{4}_{2}, x^{2}_{3}, x_{1}x^{3}_{2}, 
x^{2}_{1}x^{2}_{2}, x^{3}_{1}x_{2}, x^{2}_1 x_{3}, x^{2}_2 x_{3}, x_{1}x_{2}x_{3}] .
\end{align*}
Here,$G$ only has 
one nontrivial proper subgroup $H := \big<\mathrm{diag}(-1, -1, 1) \big>$ with nontrivial invariant subspace $V^H$. 


The algebra
$\Lambda$ is constructed as follows. 
Let $W$ be the one-dimensional irreducible representation of $G$ 
such that $g \ast w = \sigma w $ for any $w \in W$. Then 
$\Lambda = \mathrm{End}_{S}(\bigoplus\limits^{3}_{i = 0} 
(W^{\otimes i} \otimes R)^G )$. 
Let $e_i$ be the idempotent of $\Lambda$ corresponding to the 
summand $(W^{\otimes i} \otimes R)^G$,
which gives the vertex $i$ in the following diagram. We also let $e := e_0$. 
Then $\Lambda$ is isomorphic to the quiver algebra $kQ /I$, where $Q$ is
the following quiver:
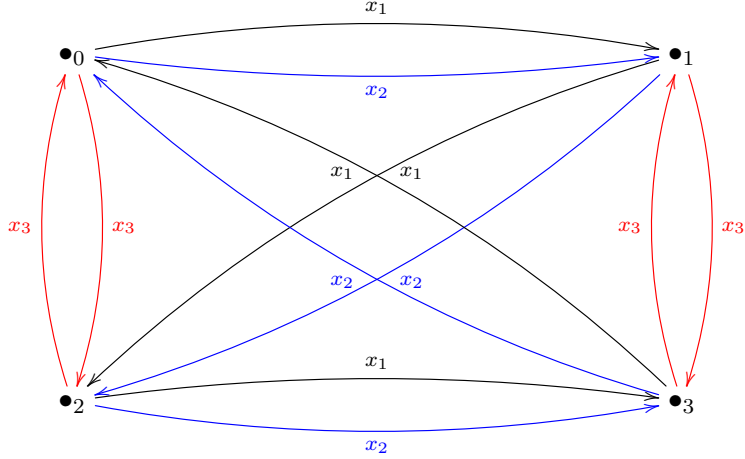
\begin{displaymath}
\xymatrix{
\bullet_{0} \ar@/^0.4cm/[rrrrrrr]^{x_{1}} 
\ar@[blue]@/_0.3cm/[rrrrrrr]_{\color{blue}{{x_{2}}}} 
\ar@[red]@/^0.4cm/[dddd]^{\color{red}{{x_{3}}}}
 &&&&&&& \bullet_{1} 
\ar@[red]@/^0.4cm/[dddd]^{\color{red}{{x_{3}}}} 
\ar@[blue]@/^0.6cm/[llllllldddd]^{\color{blue}{{x_{2}}}}  
\ar@/_0.6cm/[llllllldddd]_{x_{1}} \\ \\ \\ \\
\bullet_{2} \ar@[red]@/^0.4cm/[uuuu]^{\color{red}{{x_{3}}}} 
\ar@/^0.3cm/[rrrrrrr]^{x_{1}} 
\ar@[blue]@/_0.4cm/[rrrrrrr]_{\color{blue}{{x_{2}}}}  
&&&&&&& \bullet_{3}  
\ar@[blue]@/^0.6cm/[llllllluuuu]^{\color{blue}{{x_{2}}}} 
\ar@/_0.6cm/[llllllluuuu]_{x_{1}} 
\ar@[red]@/^0.4cm/[uuuu]^{\color{red}{{x_{3}}}} }
\end{displaymath}
The arrows of the above quiver correspond to the elements in $R$, 
and the ideal $I$ is given by the relation $\sim_{\Lambda}$ in $Q_1$ such that
$r_1 \sim_{\Lambda} r_2 $ if and only if they 
have the same sources and targets, and  
they are given by the same element in $R$.

Now, the quiver $\bar{Q}$ of $\Lambda / \Lambda e \Lambda$ is the following: 
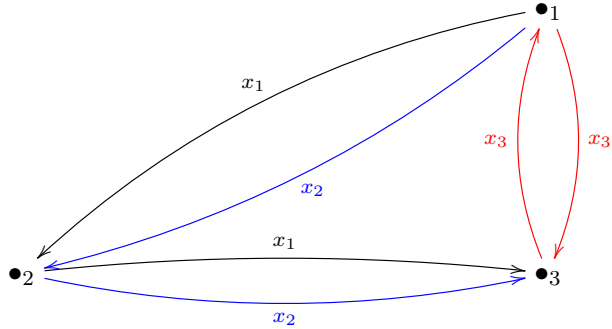
\begin{displaymath}
\xymatrix{
&&&&&& \bullet_{1} 
\ar@[red]@/^0.4cm/[ddd]^{\color{red}{{x_{3}}}} 
\ar@[blue]@/^0.5cm/[llllllddd]^{\color{blue}{{x_{2}}}}  
\ar@/_0.6cm/[llllllddd]_{x_{1}} \\ \\ \\ 
\bullet_{2} \ar@/^0.2cm/[rrrrrr]^{x_{1}} 
\ar@[blue]@/_0.4cm/[rrrrrr]_{\color{blue}{{x_{2}}}}  &&&&&&
\bullet_{3} \ar@[red]@/^0.4cm/[uuu]^{\color{red}{{x_{3}}}}  }  
\end{displaymath}
Here, $\Lambda / \Lambda e \Lambda =  k\bar{Q}/ \bar{I}$, 
where $\bar{I} := (I + (e) ) \cap k\bar{Q}$ and $(e)$ is the ideal of $kQ$ generated by $e$. 

Let $\tilde{Q}$ to be 
the following quiver:  
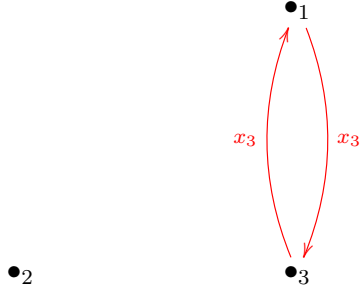
\begin{displaymath}
\xymatrix{
&&& \bullet_{1} 
\ar@[red]@/^0.4cm/[ddd]^{\color{red}{{x_{3}}}} \\ \\ \\ 
\bullet_{2}  &&& \bullet_{3} \ar@[red]@/^0.4cm/[uuu]^{\color{red}{{x_{3}}}}  }  
\end{displaymath}
and $\tilde{I} := ( I + (e) ) \cap k\tilde{Q}$. It is clear that
$$
\Lambda / \sqrt{\Lambda e \Lambda} = k \tilde{Q} / \tilde{I} 
= G/H \sharp k[V^H] \oplus e_{2} \Lambda_{0} e_{2}. 
$$ 
Moreover, we see that 
$$
\tilde{G}_0 = \big\{ (\chi_1, H ), (\chi_2, G)
\big\},
$$
where the character $\chi_i$ is corresponding to idempotent $e_i$.

By Theorem \ref{thm:smoothcatcpt2}, the generator $X$ of $\mathrm{Ker}( (-) 
\otimes_{\Lambda}^{\mathbb{L}} \Lambda e )$ is 
$$
X = \bigoplus_{0 \leq i \leq 2} G/H \sharp k[V^H](i). 
$$
Moreover, the classical generator $Y$ of $\mathrm{Ker}(\Theta)$ is
$$
Y = X \oplus (\bigoplus_{0 \leq i \leq 2}  e \Lambda (i)).
$$
\end{example}

\begin{example}
Let $V=k^3$ as in the previous example.
Now let $G\subseteq \mathrm{SL}(3,k)$
be the finite group generated by
$g=\mathrm{diag}(\sigma,\sigma^2,\sigma^3)$,
where $\sigma$ is a 6-th prime root of the unit.

With the same notations and calculations 
as in
the previous two exmaples, 
the quiver $Q$ associated to $\Lambda$ is
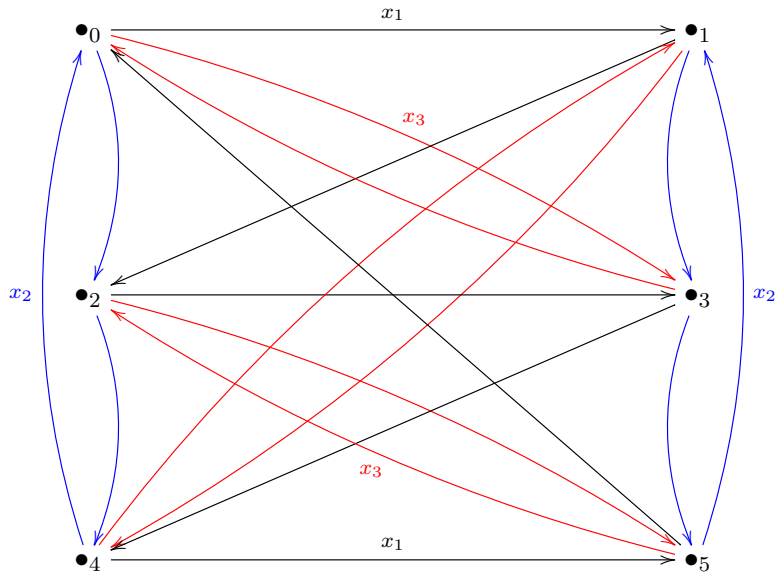
\begin{displaymath}
\xymatrix{
\bullet_{0} \ar[rrrrrrr]^{x_{1}} 
\ar@[blue]@/^0.4cm/[ddd]
\ar@[red]@/^0.4cm/[rrrrrrrddd]^{\color{red}{{x_{3}}}}
 &&&&&&& \bullet_{1} 
 \ar@[red]@/^0.6cm/[llllllldddddd]
 \ar@[blue]@/_0.4cm/[ddd] 
\ar@[black][lllllllddd]
\\ \\ \\
\bullet_{2} \ar@[blue]@/^0.4cm/[ddd]
\ar@[red]@/^0.4cm/[rrrrrrrddd]
\ar@[black][rrrrrrr]  
&&&&&&& \bullet_{3} 
\ar@[red]@/^0.4cm/[llllllluuu]
\ar@[blue]@/_0.4cm/[ddd]
\ar@[black][lllllllddd]
\\ \\ \\
\bullet_{4}
 \ar@[red]@/^0.6cm/[rrrrrrruuuuuu]
\ar@[blue]@/^0.6cm/[uuuuuu]^{\color{blue}{{x_{2}}}} 
\ar[rrrrrrr]^{x_{1}}   
&&&&&&& \bullet_{5}  
\ar@[red]@/^0.4cm/[llllllluuu]^{\color{red}{{x_{3}}}}
\ar@[black][llllllluuuuuu]
\ar@[blue]@/_0.6cm/[uuuuuu]_{\color{blue}{{x_{2}}}} 
}
\end{displaymath}
Thus the quiver $\bar Q$
of $\Lambda/\Lambda e\Lambda$ is
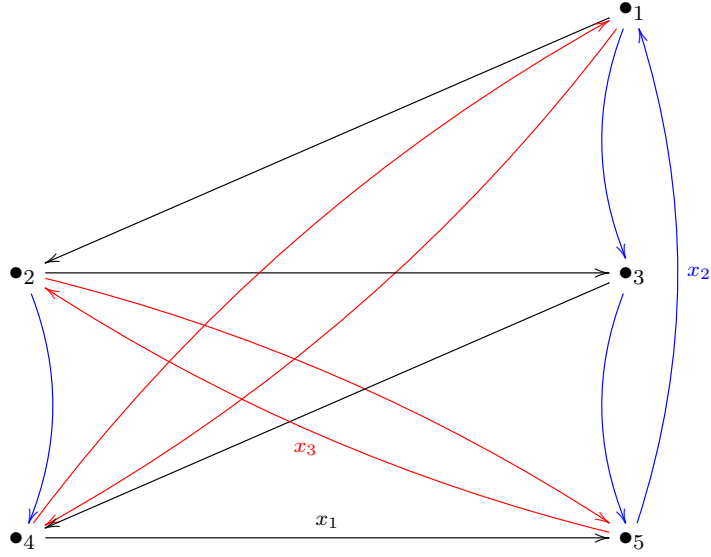
\begin{displaymath}
\xymatrix{
 &&&&&&& \bullet_{1} 
 \ar@[red]@/^0.6cm/[llllllldddddd]
 \ar@[blue]@/_0.4cm/[ddd] 
\ar@[black][lllllllddd]
\\ \\ \\
\bullet_{2} \ar@[blue]@/^0.4cm/[ddd]
\ar@[red]@/^0.4cm/[rrrrrrrddd]
\ar@[black][rrrrrrr]  
&&&&&&& \bullet_{3} 
\ar@[blue]@/_0.4cm/[ddd]
\ar@[black][lllllllddd]
\\ \\ \\
\bullet_{4}
 \ar@[red]@/^0.6cm/[rrrrrrruuuuuu] 
\ar[rrrrrrr]^{x_{1}}   
&&&&&&& \bullet_{5}  
\ar@[red]@/^0.4cm/[llllllluuu]^{\color{red}{{x_{3}}}}
\ar@[blue]@/_0.6cm/[uuuuuu]_{\color{blue}{{x_{2}}}} 
}
\end{displaymath}
and $\tilde Q$ is
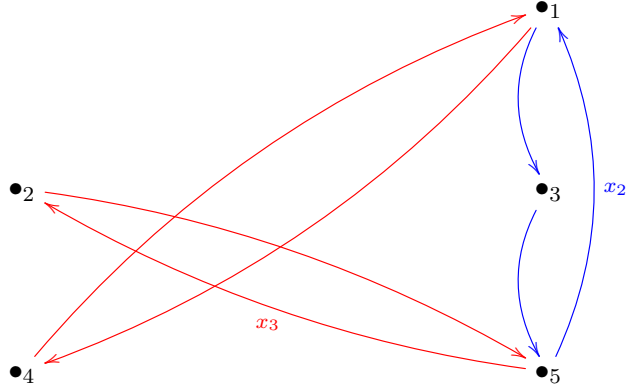
\begin{displaymath}
\xymatrix{
 &&&&&& \bullet_{1} 
 \ar@[red]@/^0.6cm/[lllllldddd]
 \ar@[blue]@/_0.4cm/[dd] 
\\ \\ 
\bullet_{2}
\ar@[red]@/^0.4cm/[rrrrrrdd]
&&&&&& \bullet_{3} 
\ar@[blue]@/_0.4cm/[dd]
\\ \\ 
\bullet_{4}
 \ar@[red]@/^0.6cm/[rrrrrruuuu]    
&&&&&& \bullet_{5}  
\ar@[red]@/^0.4cm/[lllllluu]^{\color{red}{{x_{3}}}}
\ar@[blue]@/_0.6cm/[uuuu]_{\color{blue}{{x_{2}}}} 
}
\end{displaymath}
We obtain that
$$\tilde G_0= \big\{(\chi_1, G^1),(\chi_2, G^1), (\chi_3, G^2)\big\},$$ 
where $G^1$ is the subgroup generated by $\mathrm{diag}(\sigma^2, \sigma^4, 1)$, and $G^2$ is the subgroup generated by $\mathrm{diag}(-1, 1, -1)$.

By Theorem \ref{thm:smoothcatcpt2}, the generator $X$ of $\mathrm{Ker}( (-) 
\otimes_{\Lambda}^{\mathbb{L}} \Lambda e )$ is 
$$
X = ( \bigoplus_{0 \leq i \leq 2} 
G/G^1 \sharp k[V^{G^1}](i) ) \oplus
( \bigoplus_{0 \leq i \leq 2} G/G^1 \sharp k[V^{G^1}](i) ) \oplus  
( \bigoplus_{0 \leq i \leq 2} G/G^2 \sharp k[V^{G^2}](i) ),
$$
where the graded $\Lambda/ \sqrt{\Lambda e \Lambda}$-module structure 
on $X$ is given by the graded algebra homomorphism 
$( \bigoplus_{0 \leq i \leq 2} 
\bar{\eta}_{\chi_1}^1(i) ) 
\oplus ( \bigoplus_{0 \leq i \leq 
2}\bar{\eta}_{\chi_2}^1(i) ) 
\oplus  ( \bigoplus_{0 \leq i \leq 
2} \bar{\eta}_{\chi_3}^2(i) )$. 
Moreover, the classical generator $Y$ of $\mathrm{Ker}(\Theta)$ is
$$
Y = X \oplus 
(\bigoplus_{0 \leq i \leq 2}  e \Lambda (i)).
$$
\end{example}

\end{document}